\documentclass{article}
\usepackage{graphicx}

\usepackage{amsmath}
\usepackage{amssymb}
\usepackage[margin=1.5in]{geometry}
\newcommand{\R}{\mathbb{R}}
\newcommand{\Z}{\mathbb{Z}}

\newcommand{\p}{\partial}

\newcommand{\eps}{\varepsilon}
\newcommand{\jbrac}[1]{\langle #1 \rangle}
\newcommand{\ip}[1]{\left(#1\right)}

\newcommand{\les}{\lesssim}
\newcommand{\wkly}{\rightharpoonup}
\newcommand{\supp}{\operatorname{supp}}
\newcommand{\diam}{\operatorname{diam}}
\newcommand{\Real}{\operatorname{Re}}
\newcommand{\Imag}{\operatorname{Im}}

\usepackage{amsthm}
\usepackage{xcolor}
\usepackage{hyperref}

\newtheorem{lem}{Lemma}
\newtheorem{thm}{Theorem}
\newtheorem{prop}{Proposition}
\newtheorem{remark}{Remark}
\newtheorem{cor}{Corollary}
\newtheorem{defn}{Definition}

\usepackage[backend=biber, style=nature, sorting=ynt]{biblatex}

\title{Global Existence and Time Decay for the Vlasov-Hartree System}
\author{Kieran Cavanagh\footnote{Department of Mathematics, Pennsylvania State University, University Park, PA, USA}}

\date{July 2026}

\begin{document}

\maketitle

\begin{abstract}
    The Vlasov-Hartree system is a mean-field model for a mixture of infinitely many interacting bosons and fermions where the bosons are described quantum mechanically and the fermions are described classically. This paper studies the well-posedness and dispersive properties of the Vlasov-Hartree system with initial data of arbitrary size. We prove that the Vlasov-Hartree system is globally well-posed in a low-regularity functional framework where the particle trajectories are meaningfully defined, but which includes discontinuous fermion densities. Moreover, when the interaction between the bosons and fermions is repulsive, we prove that the system exhibits dispersion in the form of time decay estimates for the particle densities and fields. When the interaction is attractive, we show that, at worst, the fields exhibit very mild growth in time.
\end{abstract}

\tableofcontents

\section{Introduction}

\subsection{The Vlasov-Hartree system}

Fermions are the particles comprising most of ordinary matter and obey the Pauli exclusion principle, which means they cannot occupy the same quantum state. Bosons, on the other hand, are particles which may occupy the same quantum state, which enables them to form a Bose-Einstein condensate in which quantum-mechanical behavior can be observed macroscopically. In this paper, we study the dynamics of a zero-temperature gas mixture of fermions and bosons interacting with each other via the interaction potential $V$.\\

In a very recent work, C\'ardenas, Miller, and Pavlovi\'c \cite{cardenas_effective_2025} studied the mean-field behavior of such mixtures, neglecting boson-boson and fermion-fermion interactions, and remarkably identified a new scaling regime in which one can describe the fermions classically and the bosons quantum-mechanically. If $N$ is the number of bosons and $M$ is the number of fermions, they consider the regime

$$\frac{M}{N} = \frac{m_B}{m_F} = \hbar,$$

where $m_B, m_F$ are the masses of the bosons and fermions respectively. In other words, the bosons are much lighter but more numerous, while the fermions are heavy but fewer. Starting from a mean-field model of boson-fermion interactions, they send $\hbar\to 0$ and $N,M\to \infty$ under the above scaling and obtain the following equations, which they call the Vlasov-Hartree system:

\begin{equation}\label{VH}
\begin{split}
    \partial_t f + v\cdot\nabla_x f + E\cdot\nabla_v f = 0, \quad E(t,x) = -\nabla_x V*|\phi(t,x)|^2,\\
    i\partial_t \phi + \tfrac12 \Delta_x \phi = (V*\rho)\phi, \quad \rho(t,x) = \int f(t,x,v)dv.
\end{split}
\end{equation}

Here, $(x,v)\in \R^3\times\R^3$ and $f$ is the phase space density of the fermions, while $\phi$ is the wavefunction for the bosons. Therefore, one can regard the Vlasov-Hartree system as a mean-field model for the dynamics of a Bose-Fermi mixture where the bosons are described quantum mechanically and the fermions are described classically. The force field $E$, which represents the force exerted by the bosons, drives the dynamics of Vlasov equation governing the fermions, while the fermions generate a potential $V*\rho$ which drives the Hartree dynamics of the bosons. We consider the case where $V$ is the Coulomb potential, i.e. 

$$V(x) = \frac{\gamma}{4\pi|x|}, \quad \gamma \in \{-1,1\}.$$

The factor $\gamma$ designates the sign of the interaction; $\gamma=1$ corresponds to repulsive interactions and $\gamma=-1$ corresponds to attractive interactions.\\

As mentioned, we assume the mixture is at zero temperature. Many of the interesting features of Bose and Fermi gases (superfluidity, superconductivity, etc.) occur at low temperatures, so it is natural to study such systems in this regime. Moreover, one expects $f$ which minimizes energy under the constraints of fixed mass and bounded fermion density $0\leq f\leq 1$ to be an indicator function; see for instance \cite{fournais_semi-classical_2018} for the case of a Fermi gas. Crucially, this means $f$ may be discontinuous, so we situate our results in a low-regularity setting in which only integrability assumptions are imposed on the initial fermion density. This introduces a number of mathematical challenges that are of independent interest and are discussed more below.\\ 

We also mention the closely related Vlasov-Poisson and Hartree equations, which are fundamental models in kinetic theory and dispersive PDE and model the mean-field behavior of plasmas (or stellar dynamics) and Bose-Einstein condensates respectively. With the same notation for $\rho$ and $V$ as above, the Vlasov-Poisson and Hartree equations respectively take the form

$$\p_t f + v\cdot\nabla_x f + E\cdot\nabla_vf = 0, \quad E = -\nabla_x V *\rho,$$
$$i\p_t\phi + \frac12\Delta\phi = (V*|\phi|^2)\phi.$$

Note that one formally obtains the Vlasov-Hartree system by coupling Vlasov-Poisson and Hartree and swapping the roles of the fermion density $\rho$ and the boson density $|\phi|^2$ in the nonlinearities. 

\subsection{Main results}

The goal of this paper is to study the well-posedness and dispersive properties of solutions to the Vlasov-Hartree system with large initial data.\\

\textbf{a. Global existence.} Since we seek solutions where the fermionic density $f$ may be discontinuous, we need to work with weak solutions. However, it is desirable to have enough regularity on $\phi$ to ensure the particle trajectories associated to the Vlasov equation is well-defined, and moreover that $\phi$ solves the Hartree equation in the mild sense (rather than merely in the distributional sense). Moreover, to ensure the weak solution is physical we look for solutions with conserved mass and energy. These requirements lead to the following notion of solution. 

\begin{defn}
    We say $(f,\phi)$ is a Lagrangian weak-mild solution to (\ref{VH}) on $[0,T]$ if 
    \begin{itemize}
        \item $(f,\phi)$ solves (\ref{VH}) in the sense of distributions;
        \item $(f,\phi)\in L^\infty([0,T];L^1_{x,v}\cap L^\infty_{x,v})\times C([0,T];H^s_x)$ for some $s>3/2$ and $f(t)$ is compactly supported for every $t\in [0,T]$;
        \item $\phi$ is a mild solution to the Hartree equation in (\ref{VH}), i.e. 
        $$\phi(t) = e^{it\Delta/2}\phi(0) - i\int_0^t e^{i(t-s)\Delta/2}[(V*\rho)\phi](s)ds;$$
        \item Mass, probability, and energy are conserved for almost every $t\in [0,T]$, i.e.
        $$\iint f(t,x,v)dxdv = \iint f(0,x,v)dxdv, \quad \int |\phi(t,x)|^2 dx = \int |\phi(0,x)|^2dx,$$
        $$\mathcal{E}(t) := \frac12 \int |\nabla_x\phi(t)|^2 dx + \frac12 \iint |v|^2 f(t)dxdv + \int (V*\rho(t))|\phi(t)|^2 dx = \mathcal{E}(0).$$
     \end{itemize}
\end{defn}

\begin{remark}
    As Lemma \ref{fieldbds} shows, the assumption $\phi\in H^s_x$ for $s>3/2$ implies that the flow map $\Phi_t:\R^6 \to \R^6$ defined by $\Phi_{t}(x,v) = (X(t,0,x,v),V(t,0,x,v))$, where
        \begin{gather}
            \begin{cases}
                \tfrac{d}{ds}X(s,t,x,v)=V(s,t,x,v) & X(t,t,x,v) = x\\
                \tfrac{d}{ds}V(s,t,x,v)=E(s,X(s,t,x,v)) & V(t,t,x,v) = v
            \end{cases}
        \end{gather}
        is a well-defined Lipschitz (in fact, $C^{1,\alpha}$ for $\alpha$ small) diffeomorphism and $f$ is constant along characteristics:
        $$f(s,X(s,t,x,v),V(s,t,x,v))=f_0(x,v) \quad \forall s,t\in[0,T], \ (x,v)\in \R^3\times \R^3,$$
        where $f_0 = f|_{t=0}$. This is the primary purpose of the regularity assumption on $\phi$ and is the reason we call such solutions Lagrangian. It also gives the solution formula
        $$f(t) = f_0 \circ\Phi^{-1}_t.$$
\end{remark}

Our first result shows that, given initial data of possibly large size, there exists a unique global Lagrangian weak-mild solution to (\ref{VH}).

\begin{thm}[Global existence]\label{globalex}
    Let $f_0\in L^\infty_{x,v}$ with compact support and $\phi_0\in H^s_x$ for some $s\in (3/2,2)$, and assume $\gamma\in \{-1,1\}$. Then for any $T<\infty$ there exists a unique Lagrangian weak-mild solution to (\ref{VH}) on $[0,T]$ with initial data $(f_0,\phi_0)$.
\end{thm}

\begin{remark}
    We stress that our notion of solution is stronger than a typical weak solution to a Vlasov-type system due to the well-definedness of the Lagrangian flow. Indeed, the solution is unique, energy is conserved, and most importantly, if $f_0 \in C^1_c$, then the method of characteristics shows that $f$ solves Vlasov in the classical sense. In this sense, Theorem \ref{globalex} bears more resemblance to results concerning global well-posedness of classical solutions to Vlasov-type systems than of weak solutions to such systems. 
\end{remark}

\begin{remark}
    We require $s<2$ in Theorem \ref{globalex} since we cannot propagate higher regularity for $\phi$ without assuming more regularity on $f$. For instance, if $f_0\in C^{k,\alpha}$, it should be possible to prove an analogous result for $s>3/2+k+\alpha$.
\end{remark}

\textbf{b. Asymptotic behavior.} Once global well-posedness is established, a natural question is how solutions behave as $t\to\infty$. Our next result concerns two key facets of asymptotic behavior:
\begin{enumerate}
    \item Decay in time of the boson and fermion densities and associated fields, which reflects the amount of dispersion in the system;
    \item Growth in time of top-order quantities, which reflects the degree to which the solution may become less regular as $t\to\infty$ and illustrates the strength of the conservation laws.
\end{enumerate}

In our functional framework, the obvious top-order norm for $\phi$ is the $H^s$ norm. However, the appropriate top order norm for $f$ is less clear. Borrowing from the analysis of the Vlasov-Poisson system, we introduce the \emph{velocity support}
$$Q(t) = \sup\{|v|:\text{there exists } \tau\in [0,t], \ x\in \R^3 \text{ such that } f(\tau,x,v)\neq 0 \}.$$
This controls the size of the largest velocities of the fermions, and is well-known to be sufficient to continue the solution for the Vlasov-Poisson and Vlasov-Maxwell systems and similar models. The proof of Theorem \ref{globalex} shows that control of $Q(t)$ is also sufficient to continue the solution in the present case. Moreover, we have the trivial estimate $\|\rho(t)\|_{L^\infty_x}\les Q(t)^3$, so $Q(t)$ controls the ``highest order" norm associated to $\rho$ within our functional framework.\\

The next result shows that the system exhibts dispersion when the interaction is repulsive, in the form of $L^p$ decay estimates for $\rho$ and $\phi$. In particular, this leads to relatively slow growth of the top-order norms. Moreover, when the interaction is attractive, we can show that the top-order norms grow at a mild algebraic rate, which illustrates the strength of the conservation laws.

\begin{thm}[Growth/decay estimates]\label{growthdecaybds}
Let $f,\phi$ be solutions to the Vlasov-Hartree system (\ref{VH}) arising from initial data satisfying the hypotheses of Theorem \ref{globalex}, and assume in addition that $x\phi_0\in L^2$. Then\\

    (i) In the repulsive case $\gamma=1$, one has the following decay estimates:
    $$\|\rho(t)\|_{L^{5/3}}\les \jbrac{t}^{-3/5}, \quad \|\phi(t)\|_{L^6}\les \jbrac{t}^{-1/2},$$
    which lead to
    $$Q(t)\les \ln^2\jbrac{t}, \quad \|\phi(t)\|_{H^s} \les \jbrac{t}^{\frac{s}{3}}\ln^{\frac{10s}{3}-4}\jbrac{t}.$$
    (ii) In the attractive case $\gamma = -1$, the following growth bounds hold:
    $$Q(t) \les \jbrac{t}\ln\jbrac{t}, \quad \|\phi(t)\|_{H^s} \les \jbrac{t}^{\frac{5s}{3}-1}\ln^{\frac{5s}{3}-2}\jbrac{t}.$$
\end{thm}

Finally, as a simple byproduct of the proof of these estimates and interpolation, we get the following growth/decay estimates for the densities and fields. In particular, this shows that the zero solution is a global $L^p$ attractor in the repulsive case.

\begin{cor}\label{moregrowthdecaybds}
There holds
    \begin{align*}
    \|E(t)\|_{L^\infty} &\les
        \begin{cases}
            \jbrac{t}^{-1}\ln\jbrac{t} & \gamma = 1,\\
            \ln\jbrac{t} & \gamma=-1,
        \end{cases}\\
        \|V*\rho(t)\|_{L^\infty} &\les 
        \begin{cases}
            \jbrac{t}^{-1/2} & \gamma=1,\\
            1 & \gamma=-1.
        \end{cases}
    \end{align*}
    Also, in the repulsive case $\gamma=1$, we get 
    $$\|\phi(t)\|_{L^\infty}\les \jbrac{t}^{\frac{3-2s}{6(s-1)}}\ln^{\frac{5s-6}{s-1}}\jbrac{t}$$
    and
    \begin{align*}
    \lim_{t\to\infty}\|\rho(t)\|_{L^p} = 0, \quad \lim_{t\to\infty}\|\phi(t)\|_{L^q} = 0,
    \end{align*}
    whenever $1<p<\infty$, $2<q\leq \infty$.
\end{cor}

\subsection{Background}

To our knowledge, our results are the first of their kind for the Vlasov-Hartree system. However, since similar results have appeared in the literature for the Vlasov-Poisson and Hartree equations separately, to motivate our results we situate our work in the context of past results for these two equations.\\

\textbf{The Vlasov-Poisson system.} On the topic of global existence, weak solutions were first shown to exist globally in \cite{arsenev_global_1975} using the compactness properties of the Poisson equation. It is well-known that control of the velocity support $Q(t)$ is sufficient to conclude global existence of classical solutions \cite{batt_global_1977}. Global bounds on $Q(t)$ and therefore global existence for classical solutions were first established by \cite{pfaffelmoser_global_1992} and notably holds in both the attractive and repulsive cases. The key tool is conservation of energy, which gives a uniform bound on $\iint |v|^2 f(t,x,v)dxdv$ and hence control over large velocities, combined with a delicate analysis in phase space. We also note the alternate proof of global existence in \cite{lions_propagation_1991} which involves the related notion of control of velocity moments rather than velocity support.\\

A different, but related, line of inquiry concerns the asymptotic behavior of solutions to Vlasov-Poisson. It is known that small data solutions exhibit modified scattering, i.e. the solution behaves asymptotically like free transport up to a logarithmic correction, see for instance \cite{choi_modified_2016},\cite{ionescu_asymptotic_2022},\cite{pankavich_asymptotic_2022}. Notably, this holds regardless of whether $\gamma=\pm 1$. In this case, the decay rates of the particle density and electric field are optimal. However, it is a longstanding open question whether one has dispersive effects for large data. Since there exist nontrivial steady states in the attractive case, for large data one only expects decay estimates for repulsive interactions. So far, in the repulsive case with large data, optimal decay rates have only been established in the context of spherical symmetry \cite{pankavich_exact_2021}, and outside of symmetry the best results are due to \cite{yang_growth_2017}, who obtains $\|E(t)\|_{L^\infty}\les \jbrac{t}^{-1/6+\eps}$, which is far from the optimal rate $\jbrac{t}^{-2}$. A key component of the latter result is the following weighted estimate for $f$:
\begin{equation}\label{weightedest}
    \iint |x-vt|^2 f(t,x,v) dxdv \les \jbrac{t}.
\end{equation}

This estimate was first established in \cite{illner_time_1996} and follows from a differential identity involving the above quantity and the potential energy, and is related to the virial theorem. The key point is that this estimate leads to the decay estimate $\|\rho(t)\|_{L^{5/3}}\les \jbrac{t}^{-3/5}$, so the solution exhibits some weak dispersion. Notably, these bounds are far from optimal.\\ 

Given that $Q(t)$ controls the size of the largest velocities and is a continuation norm, it is natural to ask how fast it grows. In general, the optimal bound is $Q(t)=O(1)$, since this is the case for free transport. However, the best known bound for large data is due to \cite{yang_growth_2017}, who obtains $\jbrac{t}^{1/8+\eps}$ in the repulsive case.\\

\textbf{The Hartree equation.} For the Hartree equation, global well-posedness is more straightforward than for Vlasov-Poisson and follows from the $H^1$ conservation law \cite{chadam_global_1975}, which holds even for attractive interactions. Regarding large time behavior, for small initial data and regardless of the sign of the interaction, one has modified scattering with the same decay rate as the linear problem \cite{hayashi_asymptotics_1998}. However, for large data optimal decay rates have not been established; to our knowledge the best results are from Hayashi and Ozawa \cite{hayashi_time_1987}, \cite{hayashi_time_1989} who obtain $\|\phi(t)\|_{L^p}\les \jbrac{t}^{-1/2}$ for $6\leq p\leq \infty$, which is far from the optimal rate $\|\phi(t)\|_{L^p}\les \jbrac{t}^{-3(\frac12-\frac1p)}$. As is the case for Vlasov-Poisson, the key idea in proving these decay rates is to control a time-weighted norm related to the free dynamics; in this case the relevant quantity is $\|J^k\phi(t)\|_{L^2}, \ k=1,2$, where $J=x+it\nabla$. Again using a conservation law related to the virial identity, Hayashi and Ozawa obtain mild growth of $\|J^k\phi(t)\|_{L^2}$, which is sufficient to prove the decay estimate for $\|\phi(t)\|_{L^p}$.

\subsection{Key contributions and ideas of the proof}

\textbf{The low-regularity setting.} Working in a low regularity setting introduces a number of technical challenges. First, we use compactness arguments since a standard Picard iteration scheme would typically require some smoothness on the Vlasov side. Second, propagating the regularity of $\phi$ demands more compactness on the Vlasov side than is typically required to construct weak solutions to Vlasov equations. For instance, to prove that $\phi$ is a mild solution or that energy is conserved, one needs compactness of $V*\rho$ in $L^\infty$ \emph{globally in space}, which one cannot obtain solely from the compactness of solutions to the Poisson equation. Importantly, this is quite different from the construction of weak solutions to Vlasov-Poisson. Instead, we need to use averaging lemmas as in the case of weak solutions to the Vlasov-Maxwell system \cite{diperna_global_1988}, which is done in Proposition \ref{strlim}. Third, as is typical for low-regularity solutions, uniqueness becomes more challenging (e.g. Yudovich's theorem for 2D Euler \cite{yudovich_non-stationary_1963}) and we need to use delicate endpoint estimates which incur logarithmic losses, leading to log-Lipschitz estimates. It is our hope that establishing a low-regularity well-posedness theory can open the door for future investigation of the qualitative dynamics of (\ref{VH}) in which low regularity is an important physical feature, such as stability of ground states.\\

\textbf{Conservation laws and global existence.} One notable feature of our results is that the conservation laws are strong enough to obtain reasonably direct bounds on $Q(t)$ and hence global existence, which is not the case for Vlasov-Poisson, where a delicate argument is required to control $Q(t)$. This suggests that the subcriticality of the present system is determined primarily by the Hartree part. We note that global existence for other couplings of dispersive equations to the Vlasov equation, such as the Vlasov-Maxwell or Vlasov-Klein-Gordon equations, is still open in three dimensions. Moreover, a key contribution of our work is that we make full use of the conservation laws to reduce the worst-case growth in time of top-order quantities. In similar situations, standard Gronwall arguments typically give exponential or double exponential growth, but we are able to show, at worst, mild algebraic growth of $\|\phi(t)\|_{H^s}$ and $Q(t)$.\\

\textbf{Dispersion and time decay.} Our results show that in the case of repulsive interactions the system always relaxes to the zero equilibrium as $t\to\infty$, illustrating that no nontrivial steady states with finite mass, energy, and spatial moments can exist (although this fact can be deduced more straightforwardly through spectral means, see e.g. \cite{simon_positive_1967}). Moreover, as a byproduct of the proof, the potential energy decays at the rate $t^{-1}$, so that all energy in the system gets converted to kinetic energy, see Proposition \ref{repdecay}. Remarkably, we can get $L^p$ convergence of the density $\rho$ for every $1<p<\infty$ due to the very slow growth bound $Q(t)\les \ln^2\jbrac{t}$, which barely misses the optimal $O(1)$ bound on $Q(t)$. Notably, this is better than any result obtained for Vlasov-Poisson with no size or symmetry assumptions on the data. Also, we obtain $\|\phi(t)\|_{L^q}\to 0$ as $t\to\infty$ for every $q\neq 2$. In both cases, the decay estimates for (nearly) the full range of $p,q$ is due to the mild growth bounds for the top order quantities.

\textbf{Ideas of the proof.} Let us provide a brief sketch of the global existence argument, the general structure of which holds in both the attractive and repulsive cases. We prove global existence alongside the construction of the solution, so in the actual proof we work with mollified versions of the quantities below. The key step is the uniform bounds on $\phi$ in $H^s$ and on $Q(t)$ in Proposition \ref{unifphibds}, which allow us to pass to the limit. By using Sobolev product estimates, estimates on the Coulomb potential, and $\rho(t)\les Q(t)^3$, one eventually can show

$$\|\phi(t)\|_{H^s} \les \|\phi_0\|_{H^s} + \int_0^t (1+Q(\tau))\|\phi(\tau)\|_{H^s}d\tau.$$

Via the method of characteristics, the velocity support $Q(t)$ can be controlled by the accumulated acceleration $\int_0^t E(s,X(s))ds$ in $L^\infty$, so the estimate is reduced to a bound on $\|E\|_{L^\infty}$. If the $L^3-L^\infty$ endpoint of the Hardy-Littlewood-Sobolev inequality held, we would have $\|E\|_{L^\infty} \les \|\phi\|_{L^6}^2$, and by Sobolev embedding $\|\phi\|_{L^6}\les \|\phi\|_{\dot{H}^1}$ and hence $Q$ would be controlled by the conserved energy. This endpoint estimate is not in fact true, but if one is willing to introduce a logarithmic factor of a higher norm, such as $\|\phi\|_{H^s}$, an endpoint Hardy-Littlewood-Sobolev estimate does in fact hold (Lemmas \ref{HLS}, \ref{loginterp}), leading to 

$$\|\phi(t)\|_{H^s} \les \|\phi_0\|_{H^s} + \int_0^t \|\phi(\tau)\|_{H^s}\ln(e+\|\phi(\tau)\|_{H^s})d\tau,$$

and therefore a nonlinear Gronwall inequality leads to the bound on $Q(t)$. Later, to obtain more optimal growth rates on these quantities, we make more explicit use of the conservation laws and additionally use commutator estimates to show that the top-order term in the estimate for $\|\phi(t)\|_{H^s}$ can be removed (Proposition \ref{Hsbd}), leading to an estimate instead of the form 

$$\|\phi(t)\|_{H^s} \les \|\phi_0\|_{H^s} + \int_0^t(1+Q(\tau))^p d\tau, \quad p<1.$$

This leads to at worst algebraic growth of $\|\phi(t)\|_{H^s}$ and $Q(t)$.\\

Regarding the decay estimates, the key tool is the following identity proved in Lemma \ref{weightedconslaw}, which is related to the ones discussed above for Vlasov-Poisson and Hartree:

$$\frac{d}{dt}\left( \frac12\iint |x-vt|^2 f(t,x,v) dxdv + \frac12 \int |(x+it\nabla)\phi|^2 dx + t^2 P(t)\right) = tP(t)$$

Here $P(t) = \int (V*\rho)|\phi|^2 dx$ is the potential energy. In Proposition \ref{repdecay} we show that this identity gives mild growth bounds on the quantities $\iint |x-vt|^2 f\ dxdv$ and $\|J\phi(t)\|_{L^2}$, leading to the decay estimates $\|\rho(t)\|_{L^{5/3}}\les \jbrac{t}^{-3/5}$ and $\|\phi(t)\|_{L^{6}}\les \jbrac{t}^{-1/2}$. By inserting these estimates into the estimates on $\|\phi(t)\|_{H^s}$ and $Q(t)$ discussed earlier, we obtain improved bounds on these quantities. Then by interpolating between the norms of $\rho,\phi$ which decay and the top-order quantities $Q(t)$, $\|\phi(t)\|_{H^s}$, we obtain the decay estimates for the full range of $p,q$ stated in Corollary \ref{moregrowthdecaybds}.

\subsection{Future questions}

\textbf{Stability of ground states.} While our result shows that repulsive interactions lead to convergence to the trivial steady state, attractive interactions can in general produce nontrivial steady states, and their stability or instability is an interesting open question. As discussed in \cite{cardenas_effective_2025} and further elaborated on in \cite{cardenas_emergence_2025}, Bose-Fermi mixtures behave differently than gases of only bosons or only fermions due to an energetic imbalance. Therefore, while steady states of finite mass and energy for the gravitational Vlasov-Poisson system and the focusing Hartree equation are known to be stable, at least in the orbital sense (see \cite{guo_stable_1999}, \cite{cazenave_orbital_1982} and related works), the situation may be different for the Vlasov-Hartree system. One of our motivations for considering low-regularity solutions is so future works can tackle such problems, since as discussed earlier, ground states for $f$ are generally indicator functions.\\

\textbf{Different potentials.} Another interesting set of questions regards the influence of the potential $V$ on the dispersive properties in the repulsive case or the well-posedness in the attractive case. For instance, when $V$ is short-range, e.g. $|x|^{-\alpha}$ for $\alpha>1$ or the screened Coulomb potential $|x|^{-1}e^{-|x|}$, one expects stronger dispersive behavior. For instance, when $V$ is homogeneous of degree $-\alpha$, a careful tracking of the proof of Lemma \ref{weightedconslaw} and Proposition \ref{repdecay} suggests that $\|\phi(t)\|_{L^6}\les \jbrac{t}^{-\alpha/2}$. However, convolution with $|x|^{-\alpha}$ for $\alpha>1$ is less smoothing than convolution with the Coulomb potential and so the local well-posedness theory for low-regularity fermion density $f$ is less clear. On the other hand, the lack of homogeneity of the screened Coulomb potential means the key dispersive identity in Lemma \ref{weightedconslaw} no longer holds.\\

Relatedly, it is known that focusing or attractive interactions can lead to finite-time blowup for Vlasov-Poisson or Hartree in certain settings. The most basic form of blowup is a virial-type argument where it is shown hat a positive quantity reaches zero in finite time. One can perform similar calculations as in the proof of Lemma \ref{weightedconslaw} to derive a virial-type identity for the present system, suggesting that blowup is possible for very singular potentials, like $|x|^{-\alpha}$ for $\alpha\geq 2$. However, as previously mentioned, the local well-posedness for very singular potentials seems to be harder.\\

\textbf{Relativistic models.} Another interesting question is whether one has global well-posedness (for repulsive interactions) for a relativistic variant of (\ref{VH}), in which $v\cdot\nabla_x$ in the Vlasov equation is replaced by $\frac{v}{\sqrt{1+|v|^2}}\cdot \nabla_x$ and $-\Delta$ in the Hartree equation is replaced by $\sqrt{-\Delta+1}$. While the model itself is of questionable physical relevance due to a lack of Lorentz invariance, its analysis may shed light on the relativistic Vlasov-Maxwell system, of which global well-posedness is a major open problem. Even global well-posedness of the relativistic Vlasov-Poisson system is open for large data, in contrast to the defocusing relativistic Hartree equation, for which global well-posedness is known \cite{lenzmann_well-posedness_2007}. This contrast makes the question of global well-posedness for the relativistic Vlasov-Hartree system intriguing.

\subsection{Notation and preliminaries}

The Fourier transform and its inverse are defined as 

$$\hat{f}(\xi) = (\mathcal{F}f)(\xi) := \int e^{-ix\cdot\xi} f(x)dx, \quad (\mathcal{F}^{-1}f)(x) := (2\pi)^{-3}\int e^{ix\cdot\xi}f(\xi)d\xi.$$

For $x\in \R^n$ we denote $\jbrac{x} = \sqrt{2+|x|^2}$; note that $\ln\jbrac{0} >0$ with this notation. We depart from the usual convention in order to more simply state the log factors in the growth/decay estimates.\\

For the Lebesgue spaces $L^p$, we will often write $L^p_x$ if $x$ is the integration variable, and similarly for $L^p_{x,v}$ if both $x,v$ are being integrated.\\

The intersection norm $L^p\cap L^q$ is defined as $\|u\|_{L^p\cap L^q} = \|u\|_{L^p}+\|u\|_{L^q}$. We will use frequently that any convex combination is controlled by the $L^p\cap L^q$ norm by Young's inequality: $\|u\|^\theta_{L^p}\|u\|_{L^q}^{1-\theta} \les \|u\|_{L^p\cap L^q}$, where $0\leq\theta\leq 1$.\\

We will frequently use the Sobolev embeddings $\dot{H}^s(\R^3) \hookrightarrow L^p(\R^3)$ where $p=\frac{6}{3-2s}$ when $s<3/2$ and $H^s(\R^3)\hookrightarrow L^\infty(\R^3)\cap C^\alpha(\R^3)$ with $\alpha=s-3/2$ when $s>3/2$.\\

Let $\Delta_j$ be the homogeneous Littlewood-Paley projections onto frequencies $|\xi|\sim 2^j$, defined as the Fourier multiplier operator $\widehat{\Delta_j f}(\xi) = \varphi(2^{-j}\xi)\hat{f}(\xi)$, where $\sum_{j\in\Z} \varphi(2^{-j}\xi) =1$ and $\varphi(\xi)$ is supported in $1/2\leq |\xi|\leq 3/2$. Then, for $1\leq p,q\leq \infty$, $s\in \R$ the homogeneous Besov space $\dot{B}^s_{p,q}$ is defined as
$$\|f\|_{\dot{B}^s_{p,q}(\R^d)} = \Big\|\big( 2^{js} \|\Delta_jf\|_{L^p}\big)_j\Big\|_{\ell^q(\Z)}.$$

\textbf{Acknowledgments.} I would like to thank Esteban C\'ardenas for a helpful conversation, whose work in \cite{cardenas_effective_2025} inspired this one. I would also like to thank my advisor, Anna Mazzucato, for support and encouragement.

\section{Basic estimates and conservation laws}

\subsection{Functional inequalities}

In this section we prove some estimates which will be used throughout the paper.\\

The first is the Hardy-Littlewood-Sobolev inequality and its endpoint. 
\begin{lem}\label{HLS}
    Let $0<\alpha<d$, $1<p<q<\infty$. The following estimate holds: 
    $$\||\cdot|^{-\alpha}*u\|_{L^q} \les \|u\|_{L^p}, \quad \frac1p = \frac1q + \frac{d-\alpha}{d}.$$
    Furthermore, when $q=\infty$, we have the following refinement:
    $$\||\cdot|^{-\alpha}*u\|_{L^\infty} \les \|u\|_{\dot{B}^0_{\frac{d}{d-\alpha},1}}.$$
\end{lem}
\begin{proof}
    The first inequality is well-known. For the second, we recall that $\dot{B}^{d/p}_{p,1}\hookrightarrow L^\infty$ for any $p$, so it suffices to show that 
    $$\big\||\cdot|^{-\alpha} * u\big\|_{\dot{B}^{d-\alpha}_{\frac{d}{d-\alpha},1}} \leq C \big\|u\big\|_{\dot{B}^{0}_{\frac{d}{d-\alpha},1}}.$$
    Therefore, let $\tilde\varphi = \sqrt{\varphi}$ and $\widehat{\tilde{\Delta_j}u} = \tilde{\varphi}(2^{-j}\cdot)\hat{u}$. It is clear that the Besov norm associated to $\tilde{\Delta}_j$ is equivalent to the one associated to $\Delta_j$, possibly after modifying $\varphi$ to ensure $\sqrt{\varphi}$ is smooth. Then by distributing $\tilde{\varphi}$ across both factors in the convolution, $\Delta_j (|\cdot|^{-\alpha}*f) = \tilde{K_j} * \tilde{\Delta}_j f$, where
    \begin{align*}
        \tilde{K_j}(x) &= \mathcal{F}^{-1} [\tilde\varphi(2^{-j}\xi)|\xi|^{-d+\alpha}](x) \\
        &= 2^{-j(d-\alpha)}2^{dj} \mathcal{F}^{-1}[\tilde\varphi(\xi)|\xi|^{-d+\alpha}](2^jx) \\
        &=: 2^{-j(d-\alpha)} 2^{dj} \tilde{K}(2^j x).
    \end{align*}

Therefore, $\|\tilde{K}_j\|_{L^1} = 2^{-j(d-\alpha)} \|\tilde{K}\|_{L^1}$, and $\tilde{K}\in L^1$ since its Fourier transform is smooth. This gives
$$\big\||\cdot|^{-\alpha} * f\big\|_{\dot{B}^{d-\alpha}_{\frac{d}{d-\alpha},1}} \leq \sum_{j\in\mathbb{Z}}2^{j(d-\alpha)} \|\tilde{K}_j\|_{L^1}\|\tilde{\Delta}_j f\|_{L^{\frac{d}{d-\alpha}}} \lesssim \sum_{j\in \mathbb{Z}} \|\tilde{\Delta}_j f\|_{L^{\frac{d}{d-\alpha}}} \lesssim \|f\|_{\dot{B}^0_{\frac{d}{d-\alpha},1}}.$$ 
\end{proof}

Motivated by the endpoint bound in Lemma \ref{HLS}, we will also need the following logarithmic interpolation inequality:
\begin{lem}\label{loginterp}
There holds
    $$\|u\|_{\dot{B}^0_{3,1}}\les \|u\|_{L^3} \ln\Big(e+\frac{\|u\|_{L^1}+\|\nabla u\|_{L^3}}{\|u\|_{L^3}} \Big).$$
\end{lem}
\begin{proof}
Using Bernstein's inequality on the low and high frequency pieces and the $L^p$ boundedness of $\Delta_j$, we obtain
    \begin{align*}
        \|u\|_{\dot{B}^0_{3,1}}&= \sum_{j<-K}  \|\Delta_j u\|_{L^3} + \sum_{|j|\leq K} \|\Delta_j u\|_{L^3} + \sum_{j>K} \|\Delta_j u\|_{L^3}\\
        &\les \|2^{2j}\chi_{j<-K}\|_{\ell^1} \big\|2^{-2j}\|\Delta_j u\|_{L^3}\big\|_{\ell^\infty} + \|\chi_{|j|\leq K}\|_{\ell^1} \big\| \|\Delta_j u\|_{L^3}\big\|_{\ell^\infty} + \|2^{-j}\chi_{j>K}\|_{\ell^1} \big\|2^{j}\|\Delta_j u\|_{L^3}\big\|_{\ell^\infty}\\
        &\les 2^{-K} \|u\|_{L^1} + K\|u\|_{L^3} + 2^{-K}\|\nabla u\|_{L^3}.
    \end{align*}
    In the final line we simply used $2^{-2K}\les 2^{-K}$.
    Choosing $$K = \max\big(1,\log_2((\|u\|_{L^1}+\|\nabla u\|_{L^3})/\|u\|_{L^3})\big)$$ and the fact that $\ln(e+|x|)\geq 1$ gives 
    \begin{align*}
        \|u\|_{\dot{B}^0_{3,1}} &\les (\|u\|_{L^1}+\|\nabla u\|_{L^3})\min\Big( \frac12, \frac{\|u\|_{L^3}}{\|u\|_{L^1}+\|\nabla u\|_{L^3}}\Big) + \|u\|_{L^3}\ln\Big(e+\frac{\|u\|_{L^1}+\|\nabla u\|_{L^3}}{\|u\|_{L^3}}\Big) \\
        &\les \|u\|_{L^3}\ln\Big(e+\frac{\|u\|_{L^1}+\|\nabla u\|_{L^3}}{\|u\|_{L^3}}\Big).
    \end{align*}
    
\end{proof}

The next lemma will be used very often.

\begin{lem}\label{potest}
    Let $0<\alpha<3$.\\
    
    (i) Denote $r_\alpha = \frac{3}{3-\alpha}$. Then as long as $1\leq p<r_\alpha<q\leq \infty$, 
    \begin{equation*}
        \||x|^{-\alpha}*u\|_{L^\infty} \les \|u\|_{L^p}^{\theta}\|u\|_{L^q}^{1-\theta}, \quad \theta = \frac{\frac{1}{r_\alpha}-\frac1q}{\frac1p - \frac1q}.
    \end{equation*}
    
    (ii) The operator $u\mapsto \nabla_x^2 V * u$ is bounded from $L^p\to L^p$ for any $1<p<\infty$.\\
    
    (iii) For any $ 3< p\leq \infty$ and $0<r<R$, we have
    \begin{equation*}
        \|\nabla_x^2 (V*u)\|_{L^\infty} \les r^{1-3/p}\|\nabla u\|_{L^p} + (1+\ln(R/r))\|u\|_{L^\infty} + R^{-3}\|u\|_{L^1}.
    \end{equation*}
\end{lem}

\begin{proof}
For the first statement, we write
    \begin{align*}
    |V * u(x)|&\leq \int_{|y|\leq R} |y|^{-\alpha} |u(x-y)|dy + \int_{|y|> R} |y|^{-\alpha} |u(x-y)|dy\\
    &\les  \||x|^{-\alpha}\|_{L^{q'}(B_R)}\|u\|_{L^q} + \||x|^{-\alpha}\|_{L^{p'}(B_R^c)}\|u\|_{L^p}\\
    &\les R^{3/q' - \alpha}\|u\|_{L^q} + R^{3/p' -\alpha} \|u\|_{L^p}.
\end{align*}
Optimizing in $R$ by choosing $R = (\|u\|_{L^q}/\|u\|_{L^p})^{[3(1/q' - 1/p')]^{-1}}$ completes the proof.\\

The second statement follows from the fact that $\p_i\p_j V(x) = x_ix_j|x|^{-5}$ is a Calderon-Zygmund kernel.\\

For the third statement, we split $\R^3$ into three regions divided by two circles of radii $r<R$. We put one derivative onto $u$ and one derivative on $V$, and then integrate by parts when $|y|>r$ to put all derivatives on $V$ in the other two regions:
    \begin{align*}
        \nabla^2V* u &= \int_{|y|<r} \nabla V(y) \nabla u(x-y)dy + \int_{|y|=r}
        \nabla V(y)u(x-y)dS(y)\\
        &\quad -\int_{r\leq |y|\leq R} \nabla^2 V(y) u(x-y)dy - \int_{|y|>R} \nabla^2 V(y)u(x-y)dy\\
        &= I_1 + I_2 + I_3+ I_4.
    \end{align*}
    
    Using $|\nabla V(y)|\les |y|^{-2}$, we get
    \begin{align*}
        |I_1| &\les \|\nabla u\|_{L^p} \big\||y|^{-2}\big\|_{L^{p'}(B_r)} \les r^{3/p'-2}\|\nabla u\|_{L^p} = r^{1-3/p}\|\nabla u\|_{L^p},\\
        |I_2|&\les \|u\|_{L^\infty}.
    \end{align*}
    Similarly, using $|\nabla^2 V(y)| \les |y|^{-3}$,
    \begin{align*}
        |I_2| \les \ln(R/r) \|u\|_{L^\infty}, \quad |I_3| \les R^{-3}\|u\|_{L^1}.
    \end{align*}
\end{proof}

Finally, we state a lemma that collects a few straightforward consequences of Lemma \ref{potest} that we will use very frequently. 
\begin{lem}\label{fieldbds}
    For any $s>3/2$, there holds
    \begin{gather*}
        \|V*\rho\|_{L^\infty} \les \|\rho\|_{L^1 \cap L^{5/3}},\\
        \|E\|_{W^{1,\infty}} \les \|\phi\|_{H^s}^2, \quad \|E\|_{L^\infty} \les \|\phi\|_{L^6}^2 \ln\Big( e+\frac{\|\phi\|_{H^s}^2}{\|\phi\|_{L^6}^2}\Big).
    \end{gather*}
\end{lem}
\begin{proof}
    The estimate for $V*\rho$ is immediate from the first statement in Lemma \ref{potest} after taking $p=1,q=5/3$ (noting that $r_1 = 3/2 \in (1,5/3)$) and using Young's inequality $a^\theta b^{1-\theta} \les a+b$.\\
    
    The first crude estimate for $E$ follows from the first statement in Lemma \ref{potest} with $p=1, q=\infty$ followed by Sobolev embedding: 
    $$\|E\|_{L^\infty}\les \||\phi|^2\|_{L^1} + \||\phi|^2\|_{L^\infty} \les \|\phi\|_{H^s}^2.$$
    The estimate for $\nabla_xE$ follows from the third statement in Lemma \ref{potest} with $$p=\frac{6}{3-2(s-1)} = \frac{6}{5-2s},$$ which is larger than $3$ since $s>3/2$ and, say, $R=2, r=1$ and Sobolev embedding:
    \begin{align*}
        \|\nabla_x E\|_{L^\infty} &\les \|\nabla_x|\phi|^2\|_{L^{\frac{6}{5-2s}}} + \||\phi|^2\|_{L^\infty} + \| |\phi|^2\|_{L^1}\\
        &\les \|\phi\|_{L^\infty}\|\nabla_x\phi\|_{L^{\frac{6}{5-2s}}} + \|\phi\|_{H^s}^2  + \|\phi\|_{L^2}^2 \\
        &\les \|\phi\|_{H^s}\|\nabla_x\phi\|_{H^{s-1}}+\|\phi\|_{H^s}^2 \les \|\phi\|_{H^s}^2.
    \end{align*}
    Finally we prove the more refined estimate for $\|E\|_{L^\infty}$. Lemma \ref{HLS} followed by Lemma \ref{loginterp}, together with the Sobolev embeddings $\dot{H}^{1/2}\hookrightarrow L^3$ and $H^s\hookrightarrow L^\infty$ give
 \begin{align*}
        \|E\|_{L^\infty} &\les \||\phi|^2\|_{\dot{B}^0_{3,1}}\\
        &\les \|\phi\|_{L^6}^2 \ln\Big(e+\frac{\|\nabla |\phi|^2\|_{L^3}}{\|\phi\|_{L^6}^2} \Big)\\
        &\les \|\phi\|_{L^6}^2 \ln\Big(e+\frac{\|\nabla\phi\|_{L^3}\|\phi\|_{L^\infty}}{\|\phi\|_{L^6}^2} \Big)\\
        &\les \|\phi\|_{L^6}^2 \ln\Big(e+\frac{\|\nabla\phi\|_{\dot{H}^{1/2}}\|\phi\|_{H^s}}{\|\phi\|_{L^6}^2} \Big)\\
        &\les \|\phi\|_{L^6}^2 \ln\Big(e+\frac{\|\phi\|_{H^s}^2}{\|\phi\|_{L^6}^2} \Big).
\end{align*} 
\end{proof}

\subsection{Conservation laws and a priori estimates}

In this section, we prove a few key identities which give a priori control over the solution. We will assume the solution is smooth; modifications for a regularized version of the system will be discussed below.\\

Define the momentum and probability current as follows:

$$j(t,x) = \int v f(t,x,v)dv, \quad \mathcal{J}(t,x) = \Imag(\overline{\phi}\nabla_x\phi).$$

Then a straightfoward computation gives the identities

$$\p_t \rho + \nabla_x\cdot j = 0, \quad \p_t |\phi|^2 + \nabla_x \cdot \mathcal{J} = 0.$$

Integrating and using the divergence theorem then yields conservation of mass and probability:
$$\|\rho(t)\|_{L^1_x} = \|f_0\|_{L^1_{x,v}}, \quad \|\phi(t)\|_{L^2_x} = \|\phi_0\|_{L^2_x}.$$

We also have conservation of energy. Let
$$K_\phi(t) = \frac12 \int|\nabla_x\phi|^2 dx, \quad P(t) = \int (V*\rho)|\phi|^2 dx,\quad K_f(t) = \frac12 \iint |v|^2f \ dv,$$
$$\mathcal{E}(t) = K_\phi(t) + P(t) + K_f(t).$$

Here, $\mathcal{E}$ is the total energy of the system, $K_\phi$ is the bosonic kinetic energy, $K_f$ is the fermionic kinetic energy, and $P$ is the potential energy.

\begin{lem}\label{consenergy}
    For smooth solutions, the energy $\mathcal{E}(t)$ is conserved: $\mathcal{E}(t)=\mathcal{E}(0)$.
\end{lem}

\begin{proof}
    Using the equation and integrating by parts, we get
\begin{align*}
    \frac12\frac{d}{dt} \int |\nabla \phi|^2 dx &= \Real \int \p_t \nabla \phi \cdot \overline{\nabla \phi} \ dx\\
    &= \Real \int -i\big( -\tfrac12 \Delta \nabla\phi + \nabla(V*\rho)\phi + (V*\rho)\nabla\phi \big)\cdot\overline{\nabla \phi} \ dx\\
    &= \Imag \int \frac12 |\nabla^2\phi|^2 + \nabla(V*\rho)\cdot\phi\overline{\nabla\phi} + (V*\rho)|\nabla\phi|^2 \ dx\\
    &= -\int \nabla(V*\rho)\cdot \mathcal{J} \ dx\\
    &= \int (V*\rho)\nabla\cdot \mathcal{J} \ dx\\
    &= -\int (V*\rho)\p_t |\phi|^2 dx,
\end{align*}
\begin{align*}
    \frac12\frac{d}{dt}\iint |v|^2 f dxdv &= -\int |v|^2 \big(v\cdot\nabla_x f + E\cdot\nabla_v f\big) dxdv\\
    &= \iint v\cdot Ef \ dxdv\\
    &= \int (V*|\phi|^2)\nabla_x\cdot j \ dx\\
    &= - \int (V*|\phi|^2)\p_t \rho \ dx\\
    &= -\int (V*\rho)\p_t|\phi|^2 dx.
\end{align*}
In the final line we used the fact that $V$ is even. Collecting terms, we arrive at the result:
\begin{gather*}
    \frac{d}{dt}\big( K_\phi(t) + K_f(t)\big) = - \int (V*\rho)\p_t|\phi|^2 + (V*\p_t\rho)|\phi|^2 dx = -\frac{d}{dt}P(t).
\end{gather*}
\end{proof}

\begin{prop}\label{atrbd}
    Conservation of energy implies $$\|\rho(t)\|_{L^{5/3}}+\|\nabla_x\phi(t)\|_{L^2}\leq C_0,$$
    where $C_0$ is a constant depending on the initial data.
\end{prop}
\begin{proof}
    In the repulsive case $\gamma=1$, all terms in the energy are nonnegative, so one gets immediate bounds on each term. In the attractive case $\gamma = -1$, one does not get an immediate bound on $K_\phi$ and $K_f$. So we first show that $P(t)$, the term of indefinite sign, can be controlled sublinearly by the other two. First we show $\|\rho(t)\|_{L^{5/3}}$ is controlled sublinearly by $K_f(t)$, which immediately gives the bound on $\|\rho(t)\|_{L^{5/3}}$ in the repulsive case:
    \begin{align*}
        \rho(t,x) &\leq \int_{|v|\leq R} f(t,x,v) \ dv + R^{-2}\int_{|v|>R} |v|^2 f(t,x,v) dv\\
        &\les \|f_0\|_{L^\infty_{x,v}} R^3 + R^{-2}\int |v|^2 f(t,x,v)  dv. 
    \end{align*}
    Optimizing in $R$ by choosing $R = \|f_0\|_{L^\infty_{x,v}}^{-1/5}\big( \int |v|^2 f(t,x,v)dv\big)^{1/5}$ gives 
    $$\rho(t,x) \les \|f_0\|_{L^\infty_{x,v}}^{2/5}\Big( \int |v|^2 f(t,x,v) dv\Big)^{3/5}.$$

    By raising both sides to the power $5/3$ and integrating, we see that 
    $$\|\rho(t)\|_{L^{5/3}}\les K_f(t)^{3/5}.$$
    
    Next, using H\"{o}lder and Lemma \ref{HLS}, 
    \begin{align*}
        |P(t)| \leq \|V*\rho\|_{L^6} \||\phi|^2\|_{L^{6/5}} \les \|\rho\|_{L^{6/5}}\|\phi\|^2_{L^{12/5}}.
    \end{align*}
    By interpolation, Sobolev embedding, and conservation of mass,
    \begin{align*}
        \|\rho\|_{L^{6/5}}&\les \|\rho\|_{L^1}^{7/12}\|\rho\|_{L^{5/3}}^{5/12} \les K_f^{1/4}\\
        \|\phi\|_{L^{5/12}}&\les \|\phi\|_{L^2}^{3/4}\|\phi\|_{L^6}^{1/4} \les K_\phi^{1/8}.
    \end{align*}
    Therefore, there exists a constant $C_0$ depending on $\|f_0\|_{L^1_{x,v}\cap L^\infty_{x,v}}$ and $\|\phi_0\|_{L^2}$ such that
    $$|P(t)| \leq C_0 K_\phi(t)^{1/4} K_f(t)^{1/4}.$$
    By applying Young's inequality twice,
    $$|P(t)| \leq C_0 (K_\phi(t)+K_f(t))^{1/2} \leq \frac{1}{2}C_0^2 + \frac12 (K_\phi(t)+K_f(t)).$$
    
    Therefore, in the attractive case we have 
    \begin{align*}
        \mathcal{E}(0) = K_\phi(t) + K_f(t) - |P(t)|\geq  \frac12(K_\phi(t)+K_f(t))-\frac12 C_0^2.
    \end{align*}
    Therefore, $K_\phi(t)+K_f(t) \leq\mathcal{E}(0) + \frac12 C_0^2$, as desired.
\end{proof}

\section{Global well-posedness}

\subsection{The regularized system}

Let $\chi^\eps(x) = \eps^{-3}\chi(\eps^{-1}x)$ with $\chi$ smooth, radial, nonnegative, and compactly supported be a standard smooth mollifier such that $\|\chi^\eps * u\|_{L^p} \leq \|u\|_{L^p}$ and $\chi^\eps * u \to u$ strongly in $L^p$ as $\eps \to 0$ for all $1\leq p< \infty$. Let $ (f^\eps_0, \phi^\eps_0) := (\chi^\eps *_v \chi^\eps *_x f_0, \chi^\eps * \phi_0)$, and let $(f^\eps,\phi^\eps)$ solve the Vlasov-Hartree equation with initial data $ (f^\eps_0, \phi^\eps_0)$ and potential $V$ replaced by $V^\eps := \chi^\eps * V$, i.e.

\begin{equation}\label{VHreg}
\begin{split}
    &\partial_t f^\eps + v\cdot\nabla_x f^\eps + E^\eps\cdot\nabla_v f^\eps = 0, \quad E^\eps = -\nabla_x V^\eps*|\phi^\eps|^2,\\
    &i\partial_t \phi^\eps + \frac12 \Delta \phi^\eps = (V^\eps*\rho^\eps)\phi^\eps, \quad \rho^\eps(t,x) = \int f^\eps(t,x,v)dv,\\
    &(f^\eps,\phi^\eps)|_{t=0} = (f^\eps_0, \phi^\eps_0), \quad V^\eps := \chi^\eps * V.
\end{split}
\end{equation}

Consider the space
$$(f^\eps,\phi^\eps)\in X_T := L^\infty([0,T];L^1\cap L^\infty)\times C([0,T];H^s).$$
By a straightforward contraction argument it is possible to prove that for each $\eps$, there is a $T_\eps>0$ such that the regularized system above is locally well-posed in $X_{T_\eps}$. Moreover, since $E^\eps$ is smooth the characteristic ODEs admit a solution $X^\eps,V^\eps$ solving

\begin{equation}
\begin{cases}
    \tfrac{d}{ds}X^\eps(s,t,x,v) = V^\eps(s,t,x,v) & X^\eps(t,t,x,v) = x,\\
    \tfrac{d}{ds}V^\eps(s,t,x,v) = E^\eps(s,X^\eps(s,t,x,v)) & V^\eps(t,t,x,v) = v.
\end{cases}
\end{equation}

\subsection{Uniform bounds and compactness}

\textbf{a. Conservation laws and a priori bounds for the regularized system.} The energy for the regularized system reads

$$\mathcal{E}^\eps(t) = \underbrace{\frac12 \int |\nabla_x\phi^\eps|^2 dx}_{K^\eps_{\phi}(t)} + \underbrace{\frac12 \iint |v|^2 f^\eps \ dxdv}_{K^\eps_{f}(t)} + \underbrace{\int (V^\eps*\rho^\eps)|\phi^\eps|^2 dx}_{P^\eps(t)}.$$

\begin{lem}\label{prelimunifbds}
    The energy $\mathcal{E}^\eps(t)$ is conserved for the regularized system, i.e. $\mathcal{E}^\eps(t) = \mathcal{E}^\eps(0)$, and the following bounds hold uniformly in $\eps$ for all $t$ for which the solution exists:
    $$\|f^\eps(t)\|_{L^p_{x,v}} + \iint |v|^2 f^\eps(t) \ dxdv + \|\rho^\eps(t)\|_{L^{5/3}} + \|\phi^\eps(t)\|_{H^1} \leq C_0,$$
    where $p\in[1,\infty]$ and $C_0$ depends only on the initial data $f_0,\phi_0$.
\end{lem}
\begin{proof}
The regularized system is nearly identical to the original, except $V$ is replaced by $V^\eps$. An inspection of the proof of Lemma \ref{consenergy} shows that only the evenness of $V$ was used, which is also true for $V^\eps$ since $\chi^\eps$ is even. Hence, energy is conserved for the regularized system.\\

Next we show that $\mathcal{E}^\eps(0)$ is uniformly bounded by the initial data, which, together with Proposition \ref{atrbd}, will imply the stated uniform bounds. Observe that if $f_0$ is supported on $|x|,|v|\leq k$, then by elementary properties of convolutions, $f^\eps_0$ is supported in the sumset $$\supp f_0 + \{|x|,|v|\leq \eps\}\subset \{|x|,|v|\leq 2k\}$$ provided (say) $\eps\leq k$, which we assume in what follows. Therefore, $K^\eps_f(0) \les (2k)^8 \|f_0\|_{L^\infty_{x,v}}$.  By the $L^p$ boundedness of mollifiers, $K^\eps_\phi(0) \leq \|\phi_0\|_{H^1}$, and therefore using the arguments of Proposition \ref{atrbd}, we also have $|P^\eps(0)|\les C_0$. As a result, conservation of energy for the regularized system combined with the arguments of Proposition \ref{atrbd} gives the stated uniform in $\eps$ bounds.

\end{proof}

\textbf{b. Uniform estimates on $\|\phi^\eps\|_{H^s}$ and $Q(t)$.} In this section we prove uniform in $\eps$ bounds which allow us to conclude the solution is global. 

\begin{prop}\label{unifphibds}
    For any $T<\infty$ such that the solution $(f^\eps,\phi^\eps)$ to the regularized Vlasov-Hartree system exists on $[0,T]$, there exists a constant $C_T$ depending only on $T$ and the initial data such that 
    $$\sup_{t\in[0,T]}\big(\|\phi^\eps(t)\|_{H^s}+Q^\eps(t)\big) \leq C_T.$$
\end{prop}

\begin{proof}

We start with the mild form of the Hartree equation. By the $H^s$ unitarity of the Schr\"odinger semigroup, the $L^p$ boundedness of mollifiers, the Kato-Ponce product estimate, Sobolev embedding, and Lemmas \ref{HLS} and \ref{fieldbds}, we get
\begin{align*}
    \|\phi^\eps(t)\|_{\dot{H}^s} &\leq \left\| e^{it\Delta/2}\phi^\eps_0 - i\int_0^t e^{i(t-s)\Delta/2}[(V^\eps*\rho^\eps)\phi^\eps](s)ds\right\|_{\dot{H}^s}\\
    &\les \|\phi_0\|_{\dot{H}^s} + \int_0^t \|V*\rho(\tau)\|_{L^\infty}\|\phi^\eps(\tau)\|_{\dot{H}^s} + \||\nabla|^s V*\rho^\eps(\tau)\|_{L^2}\|\phi^\eps(\tau)\|_{L^\infty} d\tau\\
    &\les \|\phi_0\|_{\dot{H}^s} + \int_0^t \|\rho^\eps(\tau)\|_{L^1 \cap L^{5/3}}\|\phi^\eps(\tau)\|_{\dot{H}^s} + \|\rho^\eps(\tau)\|_{L^{\frac{6}{7-2s}}}\|\phi^\eps(\tau)\|_{H^s} d\tau.
\end{align*}

Since $\|\phi^\eps(t)\|_{L^2}\leq \|\phi_0\|_{L^2}$, and $\phi_0\in H^s$, we obtain 
$$\|\phi^\eps(t)\|_{H^s} \les 1 + \int_0^t \big(\|\rho^\eps(\tau)\|_{L^1\cap L^{5/3}} + \|\rho^\eps(\tau)\|_{L^{\frac{6}{7-2s}}}\big)\|\phi^\eps(\tau)\|_{H^s} d\tau.$$

When $s$ is such that $6/(7-2s)\leq 5/3$, i.e. $s\leq 1.7$, interpolating $L^{\frac{6}{7-2s}}$ between $L^{5/3}$ and $L^1$ plus the uniform bounds on $\|\rho^\eps(\tau)\|_{L^1 \cap L^{5/3}}$ and Gronwall gives
$$\|\phi^\eps(t)\|_{H^s} \les\exp(Ct).$$

However, in the case where $s\in (1.7,2)$, we instead interpolate $L^{\frac{6}{7-2s}}$ between $L^\infty$ and $L^{5/3}$, using the trivial bound $\|\rho^\eps(t)\|_{L^\infty}\les \|f_0\|_{L^\infty_{x,v}} Q^\eps(t)^3$, to obtain
\begin{align}\label{phiepsbd}
    \|\phi^\eps(t)\|_{H^s} &\les 1 + \int_0^t (\|\rho^\eps(\tau)\|_{L^1 \cap L^{5/3}} + \|\rho^\eps(\tau)\|_{L^{5/3}}^{\frac{35-10s}{18}}\|\rho^\eps(\tau)\|_{L^\infty}^{\frac{10s-17}{18}}  )\|\phi^\eps(\tau)\|_{H^s} d\tau \notag\\
    &\les 1 + \int_0^t \big(1 + Q^\eps(\tau)^{\frac{10s-17}{6}}  \big)\|\phi^\eps(\tau)\|_{H^s} d\tau \notag\\
    &\les 1 + \int_0^t \big(1 + Q^\eps(\tau) \big)\|\phi^\eps(\tau)\|_{H^s} d\tau.
\end{align}

In the last line we used the fact that $(10s-17)/6 < 1/2$ when $s<2$.\\

Next we bound $Q^\eps(t)$. Integrating the characteristic ODE for $V^\eps$, we see that
\begin{align*}
        |v|-|V^\eps(s,t,x,v)|&\leq|V^\eps(s,t,x,v)-v|\leq \|E^\eps\|_{L^1([0,T];L^\infty)}.
\end{align*}
Therefore, $|v|>2k + |\int_0^t E^\eps(s,X^\eps(s))ds|$ implies $|V^\eps(0,t,x,v)|>2k$, so the velocity support $Q^\eps(t)$ satisfies (recalling that $f_0^\eps$ is supported in $|v|\leq 2k$)  
\begin{equation}\label{Qepsbd}
    Q^\eps(t) \leq 2k + \Big|\int_0^t E^\eps(s,X^\eps(s,t,x,v))  ds\Big| \leq 2k + t\|E^\eps\|_{L^\infty([0,t];L^\infty)}.
\end{equation}

Now we bound $E^\eps$. Lemma \ref{fieldbds} followed by Sobolev embedding and the uniform bounds on $\|\phi^\eps\|_{H^1}$ give

$$\|E^\eps\|_{L^\infty} \les \|\phi^\eps\|_{H^1}^2 \ln \Big( e+\frac{\|\phi^\eps\|_{H^s}^2}{\|\phi^\eps\|_{H^1}^2}\Big) \les \ln\big( e+\|\phi^\eps\|_{H^s}\big).$$

Therefore, (\ref{phiepsbd}) and (\ref{Qepsbd}) yield
$$\|\phi^\eps(t)\|_{H^s}\les 1 + \int_0^t \|\phi^\eps(\tau)\|_{H^s} \ln\big( e + 2k+C\tau\sup_{\sigma\in[0,\tau]}\|\phi^\eps(\sigma)\|_{H^s}\big) d\tau.$$
Setting $y(t) = \sup_{\sigma\in[0,\tau]}\|\phi^\eps(\sigma)\|_{H^s}$, we get

$$y(t) \les 1+ C_T\int_0^t y(\tau)\ln(e+y(\tau))d\tau.$$
Therefore, a standard Gronwall-type argument shows that there exists $C_T$ such that 

$$\sup_{t\in[0,T]}\|\phi^\eps(t)\|_{H^s}\leq C_T.$$
Then, we obtain the desired bound on $Q^\eps(t)$ as a consequence of the bound on $E^\eps$.

\end{proof}

As a corollary, we obtain global existence for each $\eps$ with uniform in $\eps$ bounds on the Lipschitz norm of $E^\eps$ and on the support of $f^\eps$. 

\begin{prop}
    For any $\eps$ and $T<\infty$ a solution $(f^\eps,\phi^\eps)$ to the regularized Vlasov-Hartree system exists on $[0,T]$. Moreover,
$$\sup_{t\in[0,T]}\big(\|E^\eps(t)\|_{W^{1,\infty}} + \diam(\supp f^\eps(t))\big) \leq C_T,$$
where $\diam(\supp f^\eps) \subset \R^6$ is the diameter of the support of $f^\eps$.
\end{prop}
\begin{proof}
    
If there were a maximal existence time $T_*<\infty$, by ODE theory we must either have $\lim_{t\to T_*^-}\|\phi^\eps(t)\|_{H^s} = \infty$ or the characteristics exist for only a finite time, but the first case is ruled out by Proposition \ref{unifphibds}, and the second is ruled out by the estimate
$$\|E^\eps\|_{L^\infty([0,T_*];W^{1,\infty})}\les \|\phi^\eps\|_{L^\infty([0,T_*];H^s)}^2 \les C_{T_*}$$
since by the Cauchy-Lipschitz theorem the Lagrangian flow $(X^\eps,V^\eps)$ must exist globally if $E^\eps$ is Lipschitz.\\

To prove the claim regarding the support of $f^\eps$, we need only show that the $x$ support is bounded, which follows from a similar argument used to bound the $v$ support. Integrating the characteristic equations gives
\begin{align*}
        |x|-|X^\eps(s,t,x,v)|&\leq |X^\eps(s,t,x,v)-x|\leq \int_s^t |V^\eps(\tau,t,x,v)|d\tau \leq T|v|+T\|E^\eps\|_{L^1([0,T];L^\infty)}.
    \end{align*}
    In particular, if $|x|> 2k+T|v|+T\|E^\eps\|_{L^\infty([0,T];L^\infty)}$ then $|X^\eps(0,t,x,v)|>2k$, and therefore 
    \begin{align*}
        \supp f^\eps(t) &\subset \big\{(x,v) : |x|\leq 2k+T|v|+T\|E^\eps\|_{L^1([0,T];L^\infty)}, |v|\leq 2k+\|E^\eps\|_{L^1([0,T];L^\infty)}\big\}\\
        &\subset \{(x,v) : |x|\leq 2k(1+T) + 2TC_T, |v|\leq 2k+TC_T\}.
    \end{align*} 
\end{proof}
    
\textbf{c. Weak compactness.} As a result of the above bounds, we obtain the existence of $f\in L^\infty([0,T];L^p_{x,v})$ such that, up to a subsequence,
$$f^\eps \wkly f \quad \text{weak*} \text{ in } L^\infty([0,T];L^p_{x,v}) \text{ as } \eps\to 0 \quad \forall 1\leq p\leq \infty.$$

Similarly, there exists $\phi\in C([0,T];H^s)$ such that, again up to a subsequence,
$$\phi^\eps \wkly \phi \quad \text{weak*} \text{ in } C([0,T];H^s) \text{ as } \eps\to 0.$$

By the lower semicontinuity of weak convergence, the limits $f,\phi$ enjoy the same bounds as $f^\eps,\phi^\eps$:
\begin{gather*}
    \|f\|_{L^\infty([0,T];L^p_{x,v})}\leq \|f_0\|_{L^p_{x,v}} \quad \forall 1\leq p \leq\infty;\\
    \|\phi\|_{C([0,T];H^s)}\leq C_T.
\end{gather*}

\subsection{Strong limits}

In order to prove strong convergence of the characteristic flow and that the limiting solution conserves energy, we need strong convergence \emph{globally in space}. Strong convergence on compact sets via compact Sobolev embeddings is therefore not enough. In this section we prove the following proposition:

\begin{prop}\label{strlim}
    For any $T<\infty$, as $\eps\to 0$,
    \begin{align*}
        \rho^\eps \to \rho &\quad \text{strongly in } L^1([0,T]\times \R^3); \\
        V*\rho^\eps \to V*\rho &\quad \text{strongly in } L^1([0,T]; L^\infty); \\
        \phi^\eps \to \phi &\quad \text{strongly in } L^\infty([0,T];H^{s'}) \quad \forall s'<s;\\
        E^\eps \to E &\quad \text{strongly in } L^\infty([0,T];W^{1,\infty});\\
        (X^\eps,V^\eps)\to (X,V) &\quad \text{strongly in } L^\infty([0,T];W_{x,v}^{1,\infty}).
    \end{align*}
    Where, as usual, the convergence holds up to a subsequence, and $\rho,E,X,V$ are defined in the usual way from the limits $f,\phi$ constructed in the previous section.
\end{prop}

\begin{proof}

\textbf{a. Convergence of $\rho^\eps$ and $V*\rho^\eps$.} First we prove the first two claims. We recall the following averaging lemma, which is a straightforward consequence of Theorem 1.8 in \cite{bouchut_kinetic_2000}:

\begin{lem}
    Suppose $f^n(t,x,v)$ is a bounded sequence in $L^p_{loc}([0,\infty)\times\R^6)$ for some $p>1$ solving $$\p_t f^n + v\cdot\nabla_x f^n = \nabla_v \cdot g^n$$ in the sense of distributions, with $g^n_j$ locally bounded in the space of measures $\mathcal{M}_{loc}([0,\infty) \times \R^6)$. Then for any smooth compactly supported $\psi(v)$, the sequence 
    $$\rho^n_{\psi}(t,x) := \int f^n(t,x,v)\psi(v)dv$$
    is compact in $L^q_{loc}([0,\infty)\times\R^3)$ for any $q<p$.
\end{lem}

To make use of this lemma, we fix a smooth function $\psi_R(v)$ is supported in $|v|\leq R$. Then since $f^\eps$ is bounded in $L^\infty_{loc}([0,\infty)\times\R^6)$, and $g^\eps_j = -E_j^\eps f^\eps$ is bounded in $L^1_{loc}([0,\infty)\times\R^6)$ in view of $\|E^\eps f^\eps\|_{L^1_{x,v}} \leq \|E^\eps\|_{L^\infty}\|f^\eps\|_{L^1_{x,v}} \les_T 1$, the Lemma implies that, up to a subsequence, we have $\rho^\eps_{\psi_R} \to \rho_{\psi_R}$ strongly in $L^1_{loc}([0,\infty)\times\R^3)$. Then
\begin{align*}
    \|\rho^\eps-\rho\|_{L^1([0,T]\times\R^3)} &\leq \int_0^T \int_{|x|\leq R}|\rho^\eps_{\psi_R}-\rho_{\psi_R}| dxdt + \int_0^T \int_{|x|\leq R}\int_{|v|>R} |f^\eps-f| dvdxdt\\
    &\quad +\int_0^T \int_{|x|>R}\int |f^\eps-f|dvdxdt := I_1 + I_2 + I_3.
\end{align*}

The first term $I_1$ converges to zero as $\eps\to 0$ in view of the preceding discussion. The second term is bounded by 
\begin{align*}
    I_2 \leq TR^{-2}\big(\|K^\eps_f(t)\|_{L^\infty([0,T])}+\|K_f(t)\|_{L^\infty([0,T])}\big),
\end{align*}
which can be made small if $R$ is chosen large enough in view of the boundedness of the energy. Finally, for $I_3$,
\begin{align*}
    I_3 \leq TR^{-2}\big( \||x|^2 f^\eps\|_{L^\infty([0,T];L^1_{x,v})}+\||x|^2 f\|_{L^\infty([0,T];L^1_{x,v})}\big).
\end{align*}
Similarly, by choosing $R$ large and using the boundedness of spatial moments (which can easily be seen by appealing to the support bounds on $f^\eps$ and $f$), $I_3$ can be made small.\\

As a consequence, we obtain $V*\rho^\eps \to V*\rho$ strongly in $L^1([0,T];L^\infty)$, since by Lemma \ref{potest} and the uniform bound $|\rho^\eps(t)| \les Q^\eps(t)^3 \les_T 1$, 
\begin{equation*}
    \|V*(\rho^\eps-\rho)\|_{L^1([0,T];L^\infty)} \les \|\rho^\eps-\rho\|^{2/3}_{L^1([0,T];L^1)}\|\rho^\eps-\rho\|^{1/3}_{L^1([0,T];L^\infty)} \to 0. 
\end{equation*}

\textbf{b. Convergence of $\phi^\eps$.} Next, we show that $\phi^\eps$ is Cauchy in $L^2$, which will enable us to prove the rest of the proposition. Assume that $\eta\leq \eps$. We have
\begin{align*}
    \|\phi^\eps(t) - \phi^\eta(t)\|_{L^2} &\leq \underbrace{\|(\chi^\eps-\chi^\eta)*\phi_0\|_{L^2}}_{I_0} + \underbrace{\int_0^t \big\|((V^\eps-V^\eta)*\rho^\eps)\phi^\eps\big\|_{L^2}ds}_{I_1}\\
    &+ \underbrace{\int_0^t \| (V^\eta*\rho^\eps)(\phi^\eps-\phi^\eta)\|_{L^2}ds }_{I_2} + \underbrace{\int_0^t\|V^\eta*(\rho^\eps-\rho^\eta)\phi^\eta\|_{L^2} ds}_{I_3}.
\end{align*}
For $I_0$, we use the property of mollifiers that $\|(\chi^\eps-\chi^\eta)*u\|_{L^p}\les \eps\|\nabla u\|_{L^p}$ and their $L^p$ boundedness, combined with the uniform bounds on $\phi^\eps$ in $H^1$, to obtain

$$I_0 \les \eps\|\phi_0\|_{\dot{H}^1} \les \eps.$$

For $I_1$, we use the same mollifier estimate, combined with Sobolev embedding and Lemmas \ref{HLS}, interpolation, and the uniform bounds on $\|\rho^\eps\|_{L^1\cap L^{5/3}}$ and $\|\phi^\eps\|_{H^1}$: 
\begin{align*}
    I_1 \les \eps\int_0^t \|\nabla V*\rho^\eps\|_{L^{3}}\|\phi^\eps\|_{L^{6}} ds \les \eps\int_0^t \|\rho^\eps\|_{L^{3/2}} \|\phi^\eps\|_{H^1} ds \les \eps T.
\end{align*}

For $I_2$ and $I_3$, we use the $L^p$ boundedness of mollifiers, Lemma \ref{fieldbds}, and the uniform bounds:
\begin{align*}
    I_2 &\les \int_0^t \|V*\rho^\eps\|_{L^\infty}\|\phi^\eps-\phi^\eta\|_{L^2}ds \les \int_0^t \|\rho^\eps\|_{L^1\cap L^{5/3}}\|\phi^\eps-\phi^\eta\|_{L^2}ds \les \int_0^t \|\phi^\eps-\phi^\eta\|_{L^2}ds, \\
    I_3 &\les \int_0^t \|V*(\rho^\eps-\rho^\eta)\|_{L^\infty} \|\phi^\eta\|_{L^2} ds \les \|V*(\rho^\eps-\rho^\eta)\|_{L^1([0,T];L^\infty)}.
\end{align*}

Therefore, Gronwall gives 

$$\|\phi^\eps - \phi^\eta\|_{L^\infty([0,T];L^2)} \les \big(\eps(1+T) + \|V*(\rho^\eps-\rho^\eta)\|_{L^1([0,T];L^\infty)}\big)\exp(CT).$$

The strong convergence of $V*\rho^\eps$ in $L^1([0,T];L^\infty)$ yields the strong convergence of $\phi^\eps$ to $\phi$ in $L^\infty([0,T];L^2)$ by uniqueness of weak limits. By interpolation and the uniform bounds on $\phi^\eps$ in $H^s$, this implies that $\phi^\eps$ converges strongly to $\phi$ in $L^\infty([0,T];H^{s'})$ for any $s'<s$.\\

\textbf{c. Convergence of $E^\eps$.} Let $\eta\leq \eps$. Using the mollifier estimate used previously combined with Lemmas \ref{potest} and \ref{fieldbds}, Sobolev embedding, and the uniform bounds gives
\begin{align*}
    \|E^\eps-E^\eta\|_{L^\infty} &\leq \big\|\nabla (V^\eps-V^\eta)*|\phi^\eps|^2\big\|_{L^\infty} + \big\|\nabla V^\eta * \big(|\phi^\eps|^2 - |\phi^\eta|^2\big)\big\|_{L^\infty}\\
    &\les \eps \big\|\nabla^2 V *|\phi^\eps|^2\big\|_{L^\infty} + \big\||\phi^\eps|^2 - |\phi^\eta|^2\big\|_{L^1}^{1/3}\big\||\phi^\eps|^2 - |\phi^\eta|^2\big\|_{L^\infty}^{2/3}\\
    &\les \eps \|\phi^\eps\|_{H^s}^2 + \|\phi^\eps-\phi^\eta\|_{L^2}^{1/3}\|\phi^\eps+\phi^\eta\|_{L^2}^{1/3} \big( \|\phi^\eps\|_{H^s}^2+\|\phi^\eta\|_{H^s}^2 \big)^{2/3}\\
    &\les \eps + \|\phi^\eps-\phi^\eta\|_{L^2}^{1/3}.
\end{align*}

This converges to zero as $\eps\to 0$ in view of the previous estimates. To handle $\nabla E$, we use a similar approach, except we need to use small amount of of room between the Sobolev embedding threshold $3/2$ and $s$, and the mollifier estimate $\|(\chi^\eps-\chi^\eta)*u\|_{L^\infty}\les \eps^\alpha [u]_{C^\alpha}$. To this end, let $\delta=\frac12 (s-\frac32)>0$ and $s'=\frac32+\delta$, so that $\frac32 < s' < s$. Then we use elliptic regularity, Lemma \ref{fieldbds} with $s$ replaced by $s'$, the Sobolev embeddings $H^s\hookrightarrow C^\delta$ and $H^{s'}\hookrightarrow L^\infty$, and the uniform bounds to obtain
\begin{align*}
    \|\nabla_x(E^\eps-E^\eta)\|_{L^\infty} &\leq \big\|\nabla^2 (V^\eps-V^\eta)*|\phi^\eps|^2\big\|_{L^\infty} + \big\|\nabla^2 V^\eta * \big(|\phi^\eps|^2 - |\phi^\eta|^2\big)\big\|_{L^\infty}\\
    &\les \eps^\delta \big[\nabla^2 V * |\phi^\eps|^2\big]_{C^\delta} + \big\||\phi^\eps|^2 - |\phi^\eta|^2 \big\|_{H^{s'}}\\
    &\les \eps^\delta [\phi^\eps]^2_{C^\delta} + \|\phi^\eps-\phi^\eta\|_{H^{s'}}\|\phi^\eps+\phi^\eta\|_{H^{s}}\\
    &\les \eps^\delta \|\phi^\eps\|_{H^s} + \|\phi^\eps-\phi^\eta\|_{H^{s'}}.
\end{align*}
This converges to zero in view of the uniform bounds on $\phi^\eps$ in $H^s$ and the convergence of $\phi^\eps$ in $H^{s'}$.\\

\textbf{d. Convergence of the flow.} In what follows, we abbreviate $(X,V)(s):= (X,V)(s,t,x,v)$. Integrating the characteristic equations yields
\begin{align*}
    |X^\eps(s)-X(s)| &\leq \int_s^t |V^\eps(\tau)-V(\tau)|d\tau,\\
    |V^\eps(s)-V(s)|&\leq \int_s^t \|\nabla_xE^\eps(\tau)\|_{L^\infty}|X^\eps(\tau)-X(\tau)| + \|E^\eps(\tau)-E(\tau)\|_{L^\infty} d\tau.
\end{align*}
Therefore, adding the two inequalities and using Gronwall and the uniform bound on $\|\nabla_xE^\eps\|_{L^\infty}$, we see that 

$$|X^\eps(s)-X(s)|+|V^\eps(s)-V(s)| \les_T \|E^\eps-E\|_{L^1([0,T];L^\infty)},$$
which converges to zero as $\eps\to 0$. To handle derivatives, we observe that, with $\p$ denoting either an $x$ or a $v$ gradient,
\begin{align*}
    |\p X^\eps(s)-\p X(s)| &\leq \int_s^t |\p V^\eps(\tau)-\p V(\tau)|d\tau\\
    |\p V^\eps(s)-\p V(s)|&\leq \int_s^t \|\nabla_x E^\eps(\tau)\|_{L^\infty}|\p X^\eps(\tau)-\p X(\tau)| + \|\nabla_x (E^\eps - E)(\tau)\|_{L^\infty} |\p X(\tau)| d\tau.
\end{align*}

Adding the two inequalities and using Gronwall and the uniform bound on $\|\nabla_xE^\eps\|_{L^\infty}$  gives 
$$|X^\eps(s)-X(s)|+|V^\eps(s)-V(s)| \les_T \|\nabla_x(E^\eps-E)\|_{L^1([0,T];L^\infty)}\|\p X\|_{L^\infty([0,T];L^\infty_{x,v})}.$$
This converges to zero using the convergence of $\nabla_x E^\eps$ and noting that $(x,v)\mapsto (X,V)$ is uniformly Lipschitz on $[0,T]$ since $E$ is uniformly Lipschitz on $[0,T]$.

\end{proof}

\subsection{Passing to the limit} Next we show that the obtained convergence is enough for the limits to solve each equation in an appropriate sense, and for the conservation laws to hold.\\

\textbf{a. $f$ solves Vlasov.} Given the characteristics $(X,V)$ constructed in the previous section, we will show that 
$$f(t,x,v) := f_0(X(0,t,x,v),V(0,t,x,v))$$

is a distributional solution to the Vlasov equation. Given a test function $g(t,x,v)\in C_c^\infty$, let 
$$T^\eps g = \p_t g + v\cdot\nabla_x g + E^\eps\cdot\nabla_v g,$$
$$T g = \p_t g + v\cdot\nabla_x g + E\cdot\nabla_v g.$$

Similarly, let $\Phi_t^\eps, \Phi_t:\R^6 \to \R^6$ be the flow maps, defined by
\begin{gather*}
    \Phi_t^\eps(x,v) = (X^\eps(t,0,x,v),V^\eps(t,0,x,v)),\\
    \Phi_t(x,v) = (X(t,0,x,v),V(t,0,x,v)).
\end{gather*}

Then by changing variables and using the measure preserving property of the flow, we wish to show that for all $g\in C_c^\infty$, as $\eps\to 0$,
$$\int_0^T \iint f_0^\eps(x,v) ((T^\eps g)\circ \Phi_t^\eps)(t,x,v)dxdv \to \int_0^T \iint f_0(x,v) ((Tg)\circ \Phi_t)(t,x,v)dxdv. $$

To this end, observe that since $E^\eps$ and $\Phi_t$ are uniformly Lipschitz on $[0,T]$, we have
\begin{align*}
    \|T^\eps g\|_{W^{1,\infty}_{x,v}} \les \|\p_t g\|_{W^{1,\infty}_{x,v}} + \|V(t,0)\|_{W^{1,\infty}_{x,v}}\|\nabla_x g\|_{W^{1,\infty}_{x,v}} + \|E^\eps\|_{W^{1,\infty}_{x}} \|\nabla_v g\|_{W^{1,\infty}_{x,v}} \les_T 1.
\end{align*}

By Proposition \ref{strlim}, this gives strong convergence:
\begin{align*}
    \|T^\eps g \circ \Phi^\eps_t-Tg \circ \Phi_t\|_{L^\infty_{t,x,v}} \les \|T^\eps g\|_{W^{1,\infty}_{x,v}} \|\Phi^\eps_t - \Phi_t\|_{L^\infty_{t,x,v}} + \|\Phi_t\|_{W^{1,\infty}_{x,v}}\|T^\eps g - Tg\|_{L^\infty_{t,x,v}}.
\end{align*}

\textbf{b. $\phi$ is a mild solution.} Now we show the limit $\phi$ solves the Hartree equation in the mild sense. Let
$$u(t) = e^{it\Delta/2}\phi_0 + \int_0^t e^{i(t-s)\Delta/2}[(V*\rho)\phi](s)ds.$$

Then 
\begin{align*}
    \|\phi(t)-u(t)\|_{L^2} \leq \|\phi(t)-\phi^\eps(t)\|_{L^2} + \|\phi^\eps_0-\phi_0\|_{L^2} + \int_0^t \big\|(V*\rho^\eps)\phi^\eps - (V*\rho)\phi \big\|_{L^2}ds.
\end{align*}

The first two terms converge to zero as $\eps\to 0$ by standard properties of mollifiers. The third term can be estimated as follows:
\begin{align*}
    \int_0^t\big\|(V*\rho^\eps)\phi^\eps - (V*\rho)\phi \big\|_{L^2} ds &\leq \int_0^t \|V*\rho^\eps\|_{L^\infty}\|\phi^\eps-\phi\|_{L^2} + \|V*(\rho^\eps-\rho)\|_{L^\infty}\|\phi\|_{L^2} ds\\
    &\les_T \|\phi^\eps-\phi\|_{L^\infty([0,T];L^2)} + \|V*(\rho^\eps-\rho)\|_{L^1([0,T];L^\infty)}.
\end{align*}

This converges to zero by Proposition \ref{strlim}. Hence, $\phi=u$ and so $\phi$ is a mild solution.\\

\textbf{c. Conservation of energy.} Next we show that the limits $f,\phi$ conserve energy, namely

\begin{lem}\label{limconsenergy}
    The limiting solution $(f,\phi)$ conserves energy: for any $T<\infty$ and almost every $t\in [0,T]$, we have $\mathcal{E}(t) = \mathcal{E}(0)$.
\end{lem}

\begin{proof}
    
The convergence of $K^\eps_\phi(t) \to K_\phi(t)$ as $\eps\to 0$ is a direct consequence of the strong convergence of $\phi^\eps$ in $L^\infty([0,T];H^1)$. In view of the uniform bounds on the support of $f^\eps$, there is $R=R(T)$ such that 
$$K_f^\eps(t) = \iint_{|x|,|v|\leq R} |v|^2 f^\eps(t) \ dxdv.$$
In view of the uniform bounds, for each $t$ there is a subsequence $\eps_{n^t}$ such that $f^{\eps_{n^t}}(t)\to f(t)$ weakly in $L^2$, so $K_f^{\eps_{n^t}}(t)\to K_f(t)$. Finally, to handle $P^\eps(t)$, in view of the uniform bounds we have
\begin{align*}
    |P^\eps(t)-P(t)| &\leq  \|V^\eps*\rho^\eps(t)\|_{L^\infty_x}\big\||\phi^\eps(t)|^2 - |\phi(t)|^2\big\|_{L^1_x} + \|V^\eps*\rho^\eps(t) - V*\rho(t)\|_{L^\infty_x} \|\phi(t)\|_{L^2_x}^2\\
    &\les \|\phi^\eps(t)-\phi(t)\|_{L^2_x} + \|(V^\eps-V)*\rho^\eps(t)\|_{L^\infty_x} + \|V*(\rho^\eps-\rho)(t)\|_{L^\infty_x}.
\end{align*}

The first term converges to zero using the uniform bounds and Proposition \ref{strlim}. The second term converges to zero by using standard properties of mollifiers, Lemma \ref{potest}, and the bound $\|\rho^\eps(t)\|_{L^\infty}\les Q^\eps(t)^3 \les_T 1$ from Proposition \ref{unifphibds}:
$$\|(V^\eps-V)*\rho^\eps(t)\|_{L^\infty_x} \les \eps \|\nabla V*\rho^\eps\|_{L^\infty([0,T];L^\infty)} \les \eps\|\rho^\eps\|_{L^1\cap L^\infty} \les_T \eps.$$

To handle the third term, recall that $\|V*(\rho^\eps-\rho)(t)\|_{L^\infty_x} \to 0$ in $L^1([0,T])$ by Proposition \ref{strlim}, so there exists a subsequence $\eps_n$ such that $\|V*(\rho^{\eps_n}-\rho)(t)\|_{L^\infty_x} \to 0$ for almost every $t\in[0,T]$.\\

Finally, we may conclude the argument by showing that $\mathcal{E}^\eps(0)\to \mathcal{E}(0)$, which follows immediately from the convergence properties of mollifiers in $H^1$ and $L^1\cap L^\infty$. 

\end{proof}

\subsection{Uniqueness}\label{uniqueness}

Given two solutions $(f_1,\phi_1)$ and $(f_2,\phi_2)$ emanating from the same initial data $(f_0,\phi_0)$ with associated quantities $E_i,\rho_i,X_i,V_i, \ i=1,2$ defined in the obvious way, define

$$D_X(t) = \sup_{(x,v)\in \supp f_0} |X_1(t,0,x,v)-X_2(t,0,x,v)|,$$
$$D_V(t) = \sup_{(x,v)\in \supp f_0} |V_1(t,0,x,v)-V_2(t,0,x,v)|,$$
$$D(t) = D_X(t) + D_V(t) + \|\phi_1(t)-\phi_2(t)\|_{H^1}.$$

Since $f_i(t,X_i(t,0,x,v),V_i(t,0,x,v)) = f_0(x,v)$, we are justified in only considering $(x,v)$ in the support of $f_0$. Furthermore, $X_i,V_i$ enjoy the same (time-dependent) bounds that we proved for the support of $f$. From the characteristic equations, we have 
\begin{align*}
    D_X(t) &\leq \int_0^t D_V(s)ds,\\
    D_V(t) &\leq \sup_{(x,v)\in K_t}\int_0^t |E_1(s,X_1(s,0,x,v)) - E_2(s,X_2(s,0,x,v))| ds\\
    &\les \int_0^t \|\nabla E_1(s)\|_{L^\infty} D_X(s) + \|E_1(s)-E_2(s)\|_{L^\infty} ds.
\end{align*}
Adding these two inequalities and using the uniform bound on $\|\nabla E_1\|_{L^\infty}$, we obtain
\begin{equation}\label{chardiffbd}
    D_X(t)+D_V(t) \les_T \int_0^t \|E_1(s)-E_2(s)\|_{L^\infty} + D_X(s)+D_V(s) ds.
\end{equation}

For ease of notation let $w = |\phi_1|^2 - |\phi_2|^2$. Lemma \ref{HLS} followed by Lemma \ref{loginterp} gives
\begin{align*}
    \|E_1-E_2\|_{L^\infty} &\les \|w\|_{L^3} \ln \ip{e+\frac{\|w\|_{L^1} + \|\nabla w\|_{L^3}}{\|w\|_{L^3}}}.
\end{align*}
Observe that by H\"older, Sobolev embedding, and the uniform bounds on $\phi_i$,
\begin{gather*}
    \|w\|_{L^1} \leq \|\phi_1\|_{L^2}^2+\|\phi_2\|_{L^2}^2 \les 1,\\
    \|w\|_{L^3} \leq (\|\phi_1\|_{L^6}+\|\phi_2\|_{L^6})\|\phi_1-\phi_2\|_{L^6} \les \|\phi_1-\phi_2\|_{\dot{H}^1},\\
    \|\nabla w\|_{L^3} \leq \sum_{i=1}^2 \|\phi_i\|_{L^{\frac{6}{2s-3}}}\|\nabla\phi_i\|_{L^{\frac{6}{3-2(s-1)}}} \les \sum_{i=1}^2 \|\phi_i\|_{H^{3-s}}\|\phi_i\|_{H^s} \les C_T.
\end{gather*}
Of course, in the previous line we crucially used $s>3/2$. Therefore, since $r\mapsto r\ln(e+Cr^{-1})$ is increasing,
\begin{equation}\label{Ediffbd}
    \|E_1-E_2\|_{L^\infty} \les \|\phi_1-\phi_2\|_{\dot{H}^1}\ln\Big(e + \frac{C_T}{\|\phi_1-\phi_2\|_{\dot{H}^1}} \Big).
\end{equation}

Next, using the mild formulation we employ familiar estimates (namely, Lemma \ref{fieldbds}, and the uniform bounds on $\|\phi_i\|_{L^2}$ and $\|\rho_i\|_{L^1\cap L^{5/3}}$) to obtain
\begin{align*}
    \|(\phi_1-\phi_2)(t)\|_{L^2} &\les \int_0^t \|V*(\rho_1-\rho_2)\|_{L^\infty}\|\phi_1\|_{L^2} + \|V*\rho_2\|_{L^\infty}\|\phi_1-\phi_2\|_{L^2}ds\\
    &\les \int_0^t \|V*(\rho_1-\rho_2)\|_{L^\infty} + \|\phi_1-\phi_2\|_{L^2} ds.
\end{align*}
Using the same estimates plus Sobolev embedding, the uniform bounds on $\|\phi_i\|_{H^1}$, and Lemma \ref{HLS}, we get
\begin{align*}
    \|(\phi_1-\phi_2)(t)\|_{\dot{H}^1} &\les \int_0^t \|\nabla V*(\rho_1-\rho_2)\|_{L^\infty}\|\phi_1\|_{L^2} + \|\nabla V*\rho_2\|_{L^3}\|\phi_1-\phi_2\|_{L^6} \\
    &\qquad + \|V*(\rho_1-\rho_2)\|_{L^\infty}\|\nabla\phi_1\|_{L^2} + \|V*\rho_2\|_{L^\infty}\|\nabla (\phi_1-\phi_2)\|_{L^2} ds\\
    &\leq \int_0^t \|V*(\rho_1-\rho_2)\|_{W^{1,\infty}} + \|\phi_1-\phi_2\|_{\dot{H}^1} ds.
\end{align*}
As a consequence, combining these two estimates and applying Gronwall gives
\begin{equation}\label{phidiffbd}
    \|(\phi_1-\phi_2)(t)\|_{H^1}\les_T \int_0^t \|V*(\rho_1-\rho_2)(s)\|_{W^{1,\infty}} ds.
\end{equation}

The estimates will close once we obtain Lipschitz or log-Lipschitz bounds on the above in terms of $D(t)$. To this end, we state the following lemma. Most of the ideas of its proof are well-known and are contained in (for instance) \cite{yudovich_non-stationary_1963} for 2D Euler and \cite{loeper_uniqueness_2006}, \cite{miot_uniqueness_2016} for Vlasov-Poisson, but we provide a proof below for completeness.

\begin{lem}\label{potlipbd}
    Let $\rho_1,\rho_2$ be two densities associated to solutions to the Vlasov equation with corresponding characteristic flows $X_i,V_i$. Then 
    \begin{gather*}
        \|V*(\rho_1-\rho_2)(t)\|_{L^\infty_x} \leq C_T D_X(t),\\
        \|\nabla V*(\rho_1-\rho_2)(t)\|_{L^\infty_x} \leq C_T D_X(t)\ln\Big(e+\frac{C_T}{D_X(t)}\Big).
    \end{gather*}
    Here, $C_T$ is a constant depending on the support diameter of $f_i$ up to time $T$.
\end{lem}

Granting the lemma for now, let $\omega(r) = r\ln(e+C_T/r)$, then applying the Lemma to \ref{phidiffbd}, combining \ref{Ediffbd} with \ref{chardiffbd}, and adding the resulting two inequalities and using the fact that $\ln(e+\cdot)\geq 1$ yields
\begin{align*}
    D(t) &\les_T \int_0^t \omega\big(\|\phi_1(s)-\phi_2(s)\|_{H^1}\big) + \omega\big( D_X(s)\big) + D_X(s)+D_V(s) ds\\
    &\les \int_0^t \omega\big( D(s)\big) ds.
\end{align*}

Since $\omega$ is an Osgood modulus of continuity, the Osgood lemma gives $D(t)=0$ and hence uniqueness.

\begin{proof}[Proof of Lemma \ref{potlipbd}]
    For the first claim, we change variables and use the fact that the flow is measure-preserving, which yields 
    $$V*(\rho_1-\rho_2)(t,x) = \iint \Big(\frac{1}{|x-X_1(t,0,y,v)|}-\frac{1}{|x-X_2(t,0,y,v)|}\Big) f_0(y,v)dvdy.$$
    Let 
    $$S = \{y,v:|x-X_1(t,0,y,v)| > 2D_X(t)\}.$$
    To motivate this choice, note that on $S$, 
    $$|x-X_2| \geq |x-X_1|-|X_1-X_2|> |x-X_1|-D_X(t) > \frac12|x-X_1|,$$
    where to simplify notation we wrote $X_i = X_i(t,0,y,v)$. This implies 
    \begin{align*}
        \iint_{S} \Big|\frac{1}{|x-X_1|}-\frac{1}{|x-X_2|}\Big|f_0(y,v)dvdy &\leq \iint_{S} \frac{|x-X_1|-|x-X_2|}{|x-X_1||x-X_2|}f_0(y,v)dvdy \\
        &\leq 2\|f_0\|_{L^\infty}D_X(t)\iint_{\supp(f_0)\cap S} \frac{1}{|x-X_1|^2} dvdy.
    \end{align*}
    Undoing the change of variables (again using the measure-preserving property) and recalling the bounds on the support of $f_1$, we get
    \begin{align*}
        \iint_{\supp(f_0)\cap S} \frac{1}{|x-X_1|^2} dvdy &= \iint_{\supp(f_1)\cap |x-y|> 2D_X(t)}\frac{1}{|x-y|^2}dvdy\\ 
        &\les  \iint_{\supp(f_1)}\frac{1}{|x-y|^2}dvdy \les_T 1.
    \end{align*}
    This gives the desired contribution on $S$. On $S^c$, we crudely bound and then undo the change of variables again, resulting in 
    \begin{align*}
        \iint_{S^c} \Big|\frac{1}{|x-X_1|}-\frac{1}{|x-X_2|}\Big|f_0(y,v)dvdy &\leq \|f_0\|_{L^\infty}\sum_{i=1}^2\iint_{\supp(f_0)\cap S^c} \frac{1}{|x-X_i|} dvdy \\
        &\leq \|f_0\|_{L^\infty} \sum_{i=1}^2\iint_{\supp(f_i)\cap |x-y|\leq 3D_X(t)} \frac{1}{|x-y|} dvdy \\
        &\les_T D_X(t)^2.
    \end{align*}
    In the second inequality we used the fact that on $S^c$ there holds 
    $$|x-X_2|\leq |x-X_1|+|X_1-X_2| \leq \frac32 |x-X_1| \leq 3D_X(t).$$
    Since $D_X(t)\les_T 1$, this proves the first claim.\\
    
    The second claim is very similar so we merely sketch the proof. The key difference is that, due to a logarithmic divergence, we split into three regions:
    \begin{align*}
        S_1 &= \{y,v: |x-X_1|\leq 2D_X(t)\},\\
        S_2 &= \{y,v: 2D_X(t)<|x-X_1|< 2C_T\},\\
        S_3 &= \{y,v: |x-X_1|\geq 2C_T\},
    \end{align*}
    where $C_T$ is a constant such that $D_X(t)\leq C_T$ for $t\in [0,T]$ to ensure the region is well-defined. Using the same argument used to handle the contribution from $S^c$ in the proof of the first claim, we see that the contribution from $S_1$ is $\les_T D_X(t)$, since $|\nabla V(x)|\les |x|^{-2}$. To handle the contribution from $S_2$ and $S_3$, note that by the mean value theorem, there exists $z$ on the segment between $X_1$ and $X_2$ such that
    $$\Big|\frac{x-X_1}{|x-X_1|^3}-\frac{x-X_2}{|x-X_2|^3} \Big| \les \frac{\big||x-X_1|-|x-X_2|\big|}{|x-z|^3} \les \frac{D_X(t)}{|x-X_1|^3},$$
    where in the last inequality we used
    $$|x-z|\geq |x-X_1|-|X_1-z|\geq |x-X_1|-|X_1-X_2|\geq \frac12 |x-X_1|.$$
    Therefore, we bound 
    \begin{align*}
        \iint_{S_2} \Big|\frac{x-X_1}{|x-X_1|^3}-\frac{x-X_2}{|x-X_2|^3}\Big|f_0(y,v)dvdy &\les \|f_0\|_{L^\infty}\iint_{S_2\cap \supp(f_0)} \frac{D_X(t)}{|x-X_1|^3}dvdy\\
        &\les_T D_X(t) \int_{2D_X(t)}^{2C_T} r^{-1}dr \les_T D_X(t)\ln(C_T/D_X(t)).
    \end{align*}
    Meanwhile, over $S_3$ we can more simply bound
    \begin{align*}
        \iint_{S_3} \Big|\frac{x-X_1}{|x-X_1|^3}-\frac{x-X_2}{|x-X_2|^3}\Big|f_0(y,v)dvdy &\les \|f_0\|_{L^\infty}\iint_{\supp(f_1) \cap |x-y|>2C_T} \frac{D_X(t)}{|x-y|^3}dvdy \\
        &\les_T D_X(t).
    \end{align*}
    Putting everything together and using $\ln(e+\cdot)\geq 1$, we arrive at the final bound. 
\end{proof}

\section{Dispersive behavior}

\subsection{Decay estimates}

The next lemma establishes a conservation law that will be used later to establish decay estimates in the repulsive case.

\begin{lem}\label{weightedconslaw}
    The following identity holds: 
    $$\left(\frac12\int |(x+it\nabla_x)\phi(t)|^2 dx + \frac12\iint |x-vt|^2 f(t) \ dxdv + t^2 P(t)\right) = C_1 + \int_1^t sP(s) ds,$$
    where $C_1$ is the value of the left hand side at $t=1$.
\end{lem}

The proof of the lemma will consist of two steps: first, proving the corresponding version for the regularized system, and second, showing the convergence of each term in the regularized identity.

{
\renewcommand{\thelem}{\ref*{weightedconslaw}$^\eps$}
\begin{lem}
    The following identity holds for the regularized system:
    $$\frac{d}{dt}\left(\frac12\int |(x+it\nabla_x)\phi^\eps|^2 dx + \frac12\iint |x-vt|^2 f^\eps \ dxdv + t^2 P^\eps(t)\right) =  tP^\eps(t) + Err^\eps,$$
    where 
    $$Err^\eps = t\int (C^\eps * \rho^\eps)|\phi^\eps|^2 dx, \quad C^\eps(x) = x\cdot\nabla V^\eps(x) + V^\eps(x).$$
\end{lem}
}
\begin{proof}
For the first term, we note that $J:=x+it\nabla_x$ commutes with the linear Schr\"odinger flow, so that
     \begin{align*}
         \frac12 \frac{d}{dt}\int |J\phi^\eps|^2dx &= \Imag \int J[(V^\eps*\rho^\eps)\phi^\eps]\cdot\overline{J\phi^\eps} \ dx\\
         &= t\Real\int \nabla(V^\eps*\rho^\eps)\phi^\eps \cdot \overline{J\phi^\eps} \ dx\\
         &= t\int x\cdot\nabla(V^\eps*\rho^\eps)|\phi^\eps|^2 -\Real( it\nabla(V^\eps*\rho^\eps) \cdot \phi^\eps\overline{\nabla \phi^\eps}) \ dx\\
         &= t\int x\cdot\nabla (V^\eps*\rho^\eps)|\phi^\eps|^2 \ dx + t^2 \frac{d}{dt}K_\phi^\eps(t).
     \end{align*}
     On the other hand, using the identities (which are derived by using the Vlasov equation and integrating by parts)
     $$\frac12 \frac{d}{dt} \iint |x|^2 f^\eps \ dxdv = \iint (x\cdot v) f^\eps \ dxdv,$$
     $$\frac{d}{dt}\iint (x\cdot v)f^\eps \ dx dv = 2K^\eps_f(t)+\int x\cdot E^\eps\rho^\eps \ dx,$$
     we get
     \begin{align*}
         \frac12 \frac{d}{dt}\iint |x-vt|^2 f^\eps \ dxdv &= \frac12 \frac{d}{dt}\iint (|x|^2 - 2tx\cdot v + t^2|v|^2) f^\eps \ dxdv \\
         &= -t\frac{d}{dt}\iint (x\cdot v) f^\eps \ dxdv + \frac{d}{dt}(t^2K^\eps_f(t))\\
         &= - t\int x\cdot E^\eps\rho^\eps \ dx + t^2\frac{d}{dt}K^\eps_f(t).
\end{align*}
Using $(\nabla V^\eps)(-z) = -(\nabla V^\eps)(z) $, we get
\begin{align*}
    \int x\cdot \nabla(V^\eps*\rho^\eps)|\phi^\eps|^2 dx &- \int x\cdot E^\eps\rho^\eps \ dx \\
    &= \iint x\cdot\nabla V^\eps(x-y)\rho^\eps(y)|\phi^\eps(x)|^2 dxdy + \iint y\cdot\nabla V^\eps(y-x)|\phi^\eps(x)|^2 \rho^\eps(y) dy dx\\
    &=\iint (x-y)\cdot\nabla V^\eps(x-y)\rho^\eps(y)|\phi^\eps(x)|^2 dxdv\\
    &= \int (C^\eps*\rho^\eps)|\phi^\eps|^2 dx - \iint V(x-y)\rho^\eps(y)|\phi^\eps(x)|^2dxdy\\
    &= t^{-1}Err^\eps - P^\eps(t).
\end{align*}
Therefore, using energy conservation, 
\begin{align*}
    \frac12 \frac{d}{dt}\int |J\phi^\eps|^2dx+\frac12 \frac{d}{dt}\iint |x-vt|^2 f^\eps \ dxdv &= Err^\eps -tP(t) + t^2\frac{d}{dt}(K_\phi^\eps(t)+K^\eps_f(t)) \\
    &= Err^\eps - tP^\eps(t) - t^2\frac{d}{dt}P^\eps(t) \\
    &= Err^\eps + tP^\eps(t) - \frac{d}{dt}(t^2P^\eps(t)).
\end{align*}
Moving the last term to the left hand side gives the result. 
\end{proof}

\begin{proof}[Proof of Lemma \ref{weightedconslaw}] First we will show that the error term from the previous lemma converges to zero. As a preliminary, observe that 
\begin{gather*}
    C^\eps(x) = -x\cdot\int \chi^\eps(x-y)\frac{y}{|y|^3}dy + \int y\cdot\chi^\eps(x-y)\frac{y}{|y|^3}dy = -\overline{\chi}^\eps * \nabla V,
\end{gather*}

where $\overline{\chi}^\eps(z) = z\chi^\eps(z)$. An explicit computation reveals that $\|\overline{\chi}^\eps\|_{L^p} = \eps^{-1+3/p}\|z\chi(z)\|_{L^p}$, so H\"older, Hardy-Littlewood Sobolev, Sobolev embedding, and the uniform bounds from Lemma \ref{prelimunifbds} yield 
\begin{align*}
    |\text{Err}^\eps(t)| &\les \|C^\eps*\rho^\eps\|_{L^\infty}\|\phi^\eps\|_{L^2}^2 \les \|C^\eps\|_{L^{5/2}}\|\rho^\eps\|_{L^{5/3}} \les \|\overline{\chi}^\eps\|_{L^{15/11}} \les \eps^{6/5}.
\end{align*}
Therefore, Err$^\eps(t)$ converges to zero as $\eps\to 0$.\\

Next we prove that $J\phi^\eps$ converges strongly in $L^2$ to $J\phi$. Note that $$\nabla (e^{-i|x|^2/2t}u) = -it^{-1}e^{-i|x|^2/2t}Ju$$
and more generally, for $s\geq 0$ we can define the operator $|J|^s$ via
\begin{equation}\label{Jconj}
    |J|^s = t^{s}M(t)|\nabla|^s M(-t),\quad M(t) = e^{i|x|^2/2t},
\end{equation}

so by interpolation, for some exponents $\theta_1,\theta_2>0$ we have 
    \begin{align*}
        \|J\phi^\eps - J\phi^\eta\|_{L^2} &= t\big\|\nabla [e^{-i|x|^2/2t}(\phi^\eps-\phi^\eta)] \big\|_{L^2}\\
        &\les_T \|e^{-i|x|^2/2t}(\phi^\eps-\phi^\eta)\|_{L^2}^{\theta_1} \big\||\nabla|^s \big(e^{-i|x|^2/2t}(\phi^\eps-\phi^\eta)\big)\big\|_{L^2}^{\theta_2}\\
        &\les \|\phi^\eps-\phi^\eta\|_{L^2}^{\theta_1} \big(\||J|^s\phi^\eps\|_{L^2}+\||J|^s\phi^\eta\|_{L^2}\big)^{\theta_2}.
    \end{align*}

In view of Proposition \ref{strlim}, we obtain the desired convergence of the term $\|J\phi^\eps(t)\|_{L^2}^2 \to \|J\phi(t)\|_{L^2}^2$ once we show $\||J|^s\phi^\eps\|_{L^2}$ is bounded uniformly in $\eps$. The proof of this is not much different from the proof of the uniform boundedness of $\phi^\eps$ in $H^s$, so we quickly sketch the proof. Using the mild form of the Schr\"odinger equation, (\ref{Jconj}), and the Kato-Ponce product estimate, 
\begin{align*}
    \||J|^s\phi^\eps(t)\|_{L^2} &\les \||x|^s\phi_0\|_{L^2} + \int_0^t \big\| |\nabla|^s (M(-t)\phi^\eps) \big\|_{L^2} \|V*\rho^\eps\|_{L^\infty} + \big\| |\nabla|^s V^\eps*\rho^\eps \big\|_{L^2}\|M(-t)\phi^\eps\|_{L^\infty} d\tau\\
    &\les \||x|^s\phi_0\|_{L^2} + \int_0^t \big\| |J|^s \phi^\eps \big\|_{L^2} \|\rho^\eps\|_{L^1\cap L^{5/3}} + \big\| |\nabla|^s V^\eps*\rho^\eps \big\|_{L^2}\|\phi^\eps\|_{H^s} d\tau.
\end{align*}
The second term under the integral is bounded (by a constant depending on $T$ but not $\eps$), so by using the same arguments as in the proof of the $H^s$ boundedness of $\phi$ (Proposition \ref{unifphibds}), Gronwall gives the desired bound on $|J|^s\phi^\eps$ and therefore the convergence of $J\phi^\eps$.\\

Finally, convergence of the other two terms
$$\iint |x-vt|^2 f^\eps \ dxdv \to \iint |x-vt|^2 f \ dxdv, \quad P^\eps(t)\to P(t)$$ 
follow from identical arguments used to prove convergence of the kinetic energy in the proof of Lemma \ref{limconsenergy}. 

\end{proof}

Now we use Lemma \ref{weightedconslaw} to conclude the following decay estimate in the repulsive case.

\begin{prop}\label{repdecay}
    In the repulsive case, we have
    $$P(t) \les \jbrac{t}^{-1}, \quad \int |(x+it\nabla_x)\phi|^2 dx + \iint |x-vt|^2 f \ dxdv \les \jbrac{t}$$
    and as a consequence,
    \begin{align*}
        \|\phi(t)\|_{L^6} \les \jbrac{t}^{-1/2}, \quad \|\rho(t)\|_{L^{5/3}} \les \jbrac{t}^{-3/5}.
    \end{align*}
\end{prop}
\begin{proof}
    Recalling Lemma \ref{weightedconslaw} and setting $I(t) = t^2 P(t)$, we see that
    \begin{align}\label{integratedID}
        \frac12\int |J\phi|^2 dx + \frac12\iint |x-vt|^2 f \ dxdv + I(t) = C + \int_1^t s^{-1} I(s)ds.
    \end{align}
    Using the nonnegativity of the first two terms and Gronwall, we obtain
    $$I(t) \leq C_1 e^{\int_1^t s^{-1}ds} = C_1 t \implies P(t) \les \jbrac{t}^{-1},$$
    Therefore, inserting this into (\ref{integratedID}) and crucially using the nonnegativity of $P(t)$ in the repulsive case, we get
    $$\int |J\phi|^2 dx + \iint |x-vt|^2 f \ dxdv \les C_1+t \les \jbrac{t}.$$
    To obtain the decay estimate for $\phi$, recall that
    $$\nabla_x (e^{-i|x|^2/2t}\phi) = -it^{-1}e^{-i|x|^2/2t}J\phi.$$
    Therefore, by Sobolev embedding and the growth estimate for $\int |J\phi|^2dx$,
    $$\|\phi\|_{L^6} = \|e^{-i|x|^2/2t}\phi\|_{L^6} \les \|\nabla_x(e^{-i|x|^2/2t}\phi)\|_{L^2} \leq t^{-1}\|J\phi\|_{L^2}\les \jbrac{t}^{-1/2}.$$
    
    To prove the decay estimate for $\rho$, we split:
    \begin{align*}
    	\rho(t,x)  &= \int_{|x-vt|\leq R} f \ dv + \int_{|x-vt|>R} f \ dv\\
	&\leq  \int_{|x/t-v|\leq R/t} f \ dv + R^{-2} \int_{|x-vt|>R} |x-vt|^2 f \ dv\\
	&\les (R/t)^{3}\|f_0\|_{L^\infty_{x,v}} + R^{-2} \int |x-vt|^2 f \ dv.
    \end{align*}
    Therefore, optimizing in $R$ by choosing $R=\big(t^3\|f_0\|_{L^\infty_{x,v}}^{-1}\int |x-vt|^2 f\ dv\big)^{1/5}$ gives
    $$\rho(t,x) \les t^{-6/5}\|f_0\|_{L^\infty_{x,v}}^2 \left( \int |x-vt|^2 f\ dv\right)^{3/5}.$$
    Therefore, raising both sides to the power $5/3$, integrating in $x$, and using the linear growth of $\iint |x-vt| f \ dxdv$, we get
    $$ \|\rho(t)\|_{L^{5/3}} \les \jbrac{t}^{-3/5}.$$
\end{proof}

\subsection{Growth bounds on top-order norms of $\phi$}

In this section we prove the growth bounds on $\|\phi(t)\|_{H^s}$ in terms of $Q(t)$, enabling us to close the estimates on $Q(t)$ and in turn the decay estimates for $\rho(t)$ and $E(t)$. Then, using similar arguments we also prove growth bounds for $\||J|^s\phi(t)\|_{L^2}$, enabling us to prove the decay estimate for $\phi(t)$ in $L^\infty$.

\begin{prop}\label{Hsbd}
    There holds 
    $$(i) \quad \|\phi(t)\|_{H^s}\les Q(t)^{\frac{5s}{3}-2}\jbrac{t}^{\frac{s}{3}} \quad \text{ in the repulsive case }\gamma=1,$$
    $$(ii) \quad \|\phi(t)\|_{H^s}\les Q(t)^{\frac{5s}{3}-2}\jbrac{t} \quad \text{ in the attractive case }\gamma=-1.$$
    Also, in the repulsive case, we have
    $$\||J|^s\phi(t)\|_{L^2}\les \jbrac{t}^{\frac{4s}{3}-\frac12}Q(t)^{\frac{5s}{3}-2}.$$
\end{prop}
\begin{proof}
The key idea in obtaining the improved growth rate is to remove the effect from the top order term in the energy estimate from Proposition \ref{unifphibds}, which requires use of commutator estimates. To this end, let $s=1+\delta$ where $\delta\in (1/2,1)$, and write $|\nabla|^s = R\cdot|\nabla|^\delta\nabla$, where $R$ is the vector-valued Riesz transform. Then
\begin{align}\label{commid}
    |\nabla|^s (uv) &= R\cdot |\nabla|^\delta(u\nabla v + v\nabla u)\notag \\
    &= R\cdot \big( u|\nabla|^\delta\nabla v + v|\nabla|^\delta \nabla u + C^\delta(u,\nabla v) + C^\delta(v,\nabla u)\big),
\end{align}

where
$$C^\delta(w_1,w_2) = |\nabla|^\delta (w_1w_2) - w_1|\nabla|^\delta w_2.$$

Recall the following commutator estimate from Theorem 5.1 of \cite{li_kato-ponce_2019}:
 \begin{equation}\label{commest}
     \big\|C^\delta(w_1,w_2)\big\|_{L^p} \les \big\||\nabla|^{\delta}w_1\big\|_{L^q}\|w_2\|_{L^r},
 \end{equation}
 for all $1<p<\infty, 1<q,r\leq \infty$ and $\delta>0$ satisfying $1/p = 1/q + 1/r$. By the Schr\"odinger energy estimate,
\begin{align*}
    \|\phi(t)\|_{\dot{H^s}}^2 = \|\phi_0\|_{\dot{H}^s}^2 + 2\int_0^t \Imag \jbrac{ |\nabla|^s [(V*\rho)\phi](\tau), |\nabla|^s \phi(\tau) }d\tau.
\end{align*}

We use (\ref{commid}) with $u=V*\rho$ and $v=\phi$. When all derivatives fall on $\phi$, the fact that $R$ is $L^2$-self adjoint with $R\cdot R = \text{Id}$ and the identity $|\nabla|^s = R\cdot|\nabla|^\delta\nabla$ gives
\begin{align*}
    \Imag \jbrac{ R\cdot((V*\rho)|\nabla|^\delta\nabla\phi), |\nabla|^s\phi } = \Imag\jbrac{ (V*\rho)|\nabla|^\delta\nabla\phi, |\nabla|^\delta\nabla\phi } = 0.
\end{align*}
Therefore, the top order term in the energy estimate vanishes. Using the $L^2$ boundedness of $R$, (\ref{commid}), H\"older, (\ref{commest}), Sobolev embedding, and Lemmas \ref{HLS} and \ref{potest}, and interpolation,
\begin{align*}
    \|\phi(t)\|_{\dot{H^s}}^2 &\les \|\phi_0\|_{\dot{H}^s}^2 + \int_0^t \Big(\big\||\nabla|^\delta \nabla V*\rho\big\|_{L^3} \|\phi\|_{L^6} + \big\||\nabla|^\delta V*\rho\big\|_{L^\infty} \|\nabla \phi\|_{L^2} +  \\
    & \qquad\qquad\qquad\quad +\|\nabla V*\rho\|_{L^{\frac{3}{1-\delta}}} \big\||\nabla|^\delta\phi\|_{L^{\frac{6}{1+2\delta}}}\Big) \|\phi\|_{\dot{H}^s} d\tau\\
    &\les \|\phi_0\|_{\dot{H}^s}^2 + \int_0^t \Big(\big\||\nabla|^\delta \nabla V*\rho\big\|_{L^3} + \big\||\nabla|^\delta V*\rho\big\|_{L^\infty} + \|\nabla V*\rho\|_{L^\frac{3}{1-\delta}}\Big) \|\phi\|_{H^1}\|\phi\|_{\dot{H}^s} d\tau\\
    &\les \|\phi_0\|_{\dot{H}^s}^2 + \int_0^t  \Big(\|\rho\|_{L^\frac{3}{2-\delta}} +\|\rho\|_{L^{5/3}}^{\frac{10-5\delta}{9}}\|\rho\|_{L^\infty}^{\frac{5\delta-1}{9}} + \|\rho\|_{L^\frac{3}{2-\delta}} \Big)\|\phi\|_{H^1}\|\phi\|_{\dot{H}^s} d\tau\\
    &\les \|\phi_0\|_{\dot{H}^s}^2 + \int_0^t  \|\rho\|_{L^{5/3}}^{\frac{10-5\delta}{9}}\|\rho\|_{L^\infty}^{\frac{5\delta-1}{9}} \|\phi\|_{H^1}\|\phi\|_{\dot{H}^s} d\tau.
\end{align*}
Using Gronwall, the bounds $\|\phi(t)\|_{H^1}\les 1$ and $\|\rho(t)\|_{L^{5/3}}\les \jbrac{t}^{-3/5}$, and the trivial bound $\|\rho(t)\|_{L^\infty}\les Q(t)^3$ yields 
\begin{align*}
    \|\phi(t)\|_{\dot{H}^s} &\les \|\phi_0\|_{\dot{H}^s} + \int_0^t \|\rho(\tau)\|_{L^{5/3}}^{\frac{10-5\delta}{9}} Q(\tau)^{\frac{5\delta-1}{3}} d\tau \\
    &\les \begin{cases}
        Q(t)^{\frac{5\delta-1}{3}}\jbrac{t}^{\frac{1+\delta}{3}} & \text{in the repulsive case } \gamma=1,\\
        Q(t)^{\frac{5\delta-1}{3}}\jbrac{t} & \text{in the attractive case } \gamma=-1.
    \end{cases}
\end{align*}
Now we prove the estimate on $|J|^s\phi$. We have, with $w = M(-t)\phi$,
\begin{align*}
    \big\||J|^s\phi(t)\big\|_{L^2}^2 &= \big\||x|^s\phi_0\big\|_{L^2}^2+2\int_0^t\Imag\jbrac{ |J|^s[(V*\rho)\phi](\tau),|J|^s\phi(\tau)} d\tau\\
    &= \big\||x|^s\phi_0\big\|_{L^2}^2+2\int_0^t t^{2s}\Imag\jbrac{ |\nabla|^s[(V*\rho)w](\tau),|\nabla|^s w(\tau)} d\tau.
\end{align*}

Then, estimating in the nearly the same manner as before, and using (\ref{Jconj}) and Proposition \ref{repdecay} yields 
\begin{align*}
    \||J|^s\phi(t)\|_{L^2}^2 &\les \||x|^s\phi_0\|_{L^2}^2 + \int_0^t \tau^{2s}\Big(\big\||\nabla|^\delta \nabla V*\rho\big\|_{L^3} \|w\|_{L^6} + \big\||\nabla|^\delta V*\rho\big\|_{L^\infty} \|\nabla w\|_{L^2} +  \\
    & \qquad\qquad\qquad\quad +\|\nabla V*\rho\|_{L^{\frac{3}{1-\delta}}} \big\||\nabla|^\delta w\|_{L^{\frac{6}{1+2\delta}}}\Big) \||\nabla|^s w\|_{L^2} d\tau\\
    &\les 1 + \int_0^t \tau^s \Big(\big\||\nabla|^\delta \nabla V*\rho\big\|_{L^3} + \big\||\nabla|^\delta V*\rho\big\|_{L^\infty} + \|\nabla V*\rho\|_{L^\frac{3}{1-\delta}}\Big) \||\nabla|w\|_{L^2}\||J|^s \phi\|_{L^2} d\tau\\
    &\les 1 + \int_0^t  \tau^{s-1} \Big(\|\rho\|_{L^\frac{3}{2-\delta}} +\|\rho\|_{L^{5/3}}^{\frac{10-5\delta}{9}}\|\rho\|_{L^\infty}^{\frac{5\delta-1}{9}} + \|\rho\|_{L^\frac{3}{2-\delta}} \Big)\|J\phi\|_{L^2}\||J|^s\phi\|_{L^2} d\tau\\
    &\les 1 + \int_0^t  \tau^{s-\frac12} \|\rho\|_{L^{5/3}}^{\frac{10-5\delta}{9}}\|\rho\|_{L^\infty}^{\frac{5\delta-1}{9}} \||J|^s\phi\|_{L^2} d\tau\\
    &\les 1 + \int_0^t \tau^{s-\frac12}\jbrac{\tau}^{-1+\frac{s}{3}}Q(\tau)^{\frac{5s}{3}-2}\||J|^s\phi\|_{L^2} d\tau.
\end{align*}
Therefore, Gronwall yields 
$$\||J|^s\phi(t)\|_{L^2} \les \jbrac{t}^{\frac{4s}{3}-\frac12}Q(t)^{\frac{5s}{3}-2}.$$

\end{proof}

\subsection{Bounds on the velocity support}

Using the improved growth bounds on $\phi$ in $H^s$ from Proposition \ref{Hsbd}, we can obtain more refined growth estimates on $Q(t)$ via a nonlinear Gronwall inequality.

\begin{prop}\label{Qbd}
There holds
    $$(i) \quad Q(t) \les \ln^2\jbrac{t} \quad \text{in the repulsive case},$$
    $$(ii) \quad Q(t) \les \jbrac{t}\ln\jbrac{t} \quad \text{in the attractive case}.$$
\end{prop}
\begin{proof}
    First, recall from Lemma \ref{fieldbds} that
    $$\|E\|_{L^\infty} \les \|\phi\|_{L^6}^2\ln\Big(e+\frac{\|\phi\|_{H^s}^2}{\|\phi\|_{L^6}^2}\Big).$$
    In the repulsive case, we use Proposition \ref{repdecay}, Proposition \ref{Hsbd}, and the fact that $x\mapsto x\ln(e+1/x)$ is increasing to obtain 
    \begin{equation}\label{repEbd}
        \|E(t)\|_{L^\infty}\les \jbrac{t}^{-1}\ln(e+CQ(t)^{\frac{5s-6}{3}}\jbrac{t}^{\frac{s}{3}})\les (e+t)^{-1}\big(\ln(e+t) + \ln(e+Q(t))\big).
    \end{equation}
    The attractive case can be handled similarly after using Sobolev embedding $H^1\hookrightarrow L^6$ and Lemma \ref{atrbd} and Proposition \ref{Hsbd}:
    \begin{equation}\label{atrEbd}
        \|E(t)\|_{L^\infty}\les \ln(e+CQ(t)^{\frac{5s-6}{3}}\jbrac{t})\les \ln(e+t) + \ln(e+Q(t)).
    \end{equation}
    
Now we bound $Q(t)$.\\ 

\emph{(i) The repulsive case.} Recalling from the support estimate in the proof of the existence theorem that $|v|>k + |\int_0^t E(s,X(s))ds|$ implies $|V(0,t,x,v)|>k$, the velocity support is bounded by  
\begin{align*}
    Q(t) &\leq k + \Big|\int_0^t E(s,X(s,t,x,v))  ds\Big| \leq k + C\int_0^t  \frac{\ln(e+s)+\ln(e+Q(s))}{e+s}ds.
\end{align*}
Let $y(t)$ be the rightmost side of the inequality above, so that $Q(t)\leq y(t)$. Then 
$$y'(t) = C(e+t)^{-1}\big(\ln(e+t) + \ln(e+Q(t))\big) \leq C(e+t)^{-1}\big(\ln(e+t) + \ln(e+y(t))\big).$$

We first ``homogenize" the differential inequality above in order to recast it in the form $w'(t)\leq a(t)\ln(e+w(t))$, which is more amenable to Gronwall-type estimates. To this end let
$$w(t) = y(t)-2C\int_0^t\frac{\ln(e+s)}{e+s}ds = y(t) - C(\ln^2(e+s)-1),$$
so that, using the sublinearity of $x\mapsto \ln(e+x)$,
\begin{align*}
    w'(t) &= y'(t) -\frac{2C\ln(e+t)}{e+t}\\
    &\leq \frac{C}{e+t}\left( -\ln(e+t) + \ln\big(e+w(t) + C(\ln^2(e+t)-1)\big)\right)\\
    &\leq \frac{C}{e+t}\left(-\ln(e+t) + \ln\left(e+C(\ln^2(e+t)-1)\right)
     + \ln(e+w(t))\right)\\
     &\leq C\frac{\ln(e+w(t))}{e+t}.
\end{align*}
In the last line we used $\ln\left(e+C(\ln^2(e+t)-1)\right) \leq \ln(e+t)$, which can be seen by noting that the two functions agree at $t=0$ and are monotone increasing but the former grows slower. Dividing this inequality by $\ln(e+w(t))$ and integrating, we obtain
$$F(w(t)) \leq F(w(0)) + C\int_0^t \frac{ds}{e+s} \leq F(k) + C\ln(e+t),$$
where
$$F(u) = \int_0^u \frac{dr}{\ln(e+r)}.$$
To bound $F^{-1}$, observe that $F(u) \geq \frac{u}{\ln(e+u)}$, and as a consequence,
\begin{align*}
    F(3u\ln(e+u)) &\geq \frac{3u\ln(e+u)}{\ln(e+3u\ln(e+u))}\\
    &\geq \frac{3u\ln(e+u)}{\ln(e+u)+\ln(e+3\ln(e+u))} \geq u.
\end{align*}

In the final inequality used the fact that $\ln(e+3\ln(e+u)) \leq 2\ln(e+u)$, which holds at $u=0$ in view of $\ln(e+3) \approx 1.74 \leq 2$ and the two functions are monotone increasing with the former growing slower, so it holds for all $t$. Since $F$ and $F^{-1}$ are monotone increasing, this implies $F^{-1}(u)\leq 3u\ln(e+u)$. Therefore,

\begin{align*}
    w(t) &\leq F^{-1}\ip{ F(k) + C\ln(e+t) }\\
    &\leq 3(F(k) + C\ln(e+t))\ln(e+F(k) + C\ln(e+t))\\
    &\les \ln\jbrac{t}\ln\jbrac{\ln\jbrac{t}}.
\end{align*}
Hence $y(t) \les \ln^2\jbrac{t} +  \ln\jbrac{t}\ln\jbrac{\ln\jbrac{t}} \les \ln^2\jbrac{t}$, as desired.\\

\emph{(ii) The attractive case.} Here, our differential inequality takes the form 
$$y'(t) \leq C\big(\ln(e+t) + \ln(e+y(t)) \big),$$
where 
$$y(t) = k + C\int_0^t  \ln(e+s)+\ln(e+Q(s)) ds.$$
We ``homogenize" in a slightly different manner by choosing 
$$w(t) = \frac{y(t)}{\ln(e+t)}.$$
This leads to
$$w'(t) = \frac{y'(t)}{\ln(e+t)
} - \frac{y(t)}{(e+t)\ln^2(e+t)}\leq C+C\frac{\ln(e+y(t))}{\ln(e+t)}-\frac{w(t)}{(e+t)\ln(e+t)}.$$

Note that $\ln(u)\leq u$, so with $u=\frac{e+y(t)}{C(e+t)\ln(e+t)}$, we get
\begin{align*}
    \ln(e+y(t)) &\leq \frac{e+y(t)}{C(e+t)\ln(e+t)} + \ln(C(e+t)\ln(e+t)) \\
    &= \frac{e}{C(e+t)} + \frac{w(t)}{C(e+t)} + \ln(C(e+t)\ln(e+t)).
\end{align*}
Crucially, this cancels the term involving $w(t)$ on the right hand side:
\begin{align*}
    w'(t) \leq C + \frac{e}{(e+t)\ln(e+t)} + C\frac{\ln(C(e+t)\ln(e+t))}{\ln(e+t)} \leq C'.
\end{align*}
Therefore, integrating gives
$$w(t) \leq w(0) + C't \implies Q(t)\leq y(t) \les \jbrac{t}\ln\jbrac{t}.$$

\end{proof}

\begin{remark}
    It might be possible to reduce the power of the logarithmic growth by refining the estimate of the high frequency part in Lemma \ref{loginterp}, which is currently not optimal.
\end{remark}

\subsection{Proof of Corollary \ref{moregrowthdecaybds}}

\emph{Proof of Corollary \ref{moregrowthdecaybds}.} The estimates for $E$ follow from combining the estimate (\ref{repEbd}) in the repulsive case and (\ref{atrEbd}) in the attractive case with the bounds on $Q(t)$ from Proposition \ref{Qbd}. Similarly, the estimates for $V*\rho$ follow from Lemma \ref{potest} with $q=5/3$ and $p=1$ combined with Proposition \ref{repdecay} in the repulsive case and Proposition \ref{atrbd} in the attractive case.\\

Next we prove the $L^\infty$ decay estimate for $\phi(t)$. Using the Gagliardo-Nirenberg interpolation inequality, (\ref{Jconj}), and Propositions \ref{repdecay} and \ref{Hsbd},
\begin{align*}
    \|\phi(t)\|_{L^\infty}=\|M(-t)\phi(t)\|_{L^\infty} &\les \|M(-t)\phi(t)\|_{\dot{H}^s}^\theta \|M(-t)\phi(t)\|_{L^6}^{1-\theta}\\
    &\les t^{-s\theta-\frac12(1-\theta)}\||J|^s\phi(t)\|_{L^2}^\theta\\ 
    &\les \jbrac{t}^{\theta(\frac{s}{3}+\frac12)-\frac12}\ln^{2\theta(\frac{5s}{3}-2)}\jbrac{t},
\end{align*}
where
$$0=\theta\Big(\frac12-\frac{s}{3}\Big)+\frac{1-\theta}{6} \implies \theta = \frac{1}{2(s-1)}.$$
As a consequence,
$$\|\phi(t)\|_{L^\infty}\les \jbrac{t}^{\frac{3-2s}{6(s-1)}}\ln^{\frac{5s-6}{s-1}}\jbrac{t}.$$

Finally, the decay to zero of $L^p$ norms of $\rho$ and $\phi$ in the repulsive case follow from interpolation and conservation of mass/probability. Indeed, for the decay estimate of $\rho$, if $p<5/3$ we interpolate between $\|\rho\|_{L^1}$ and $\|\rho(t)\|_{L^{5/3}}$; the former is conserved and the latter decays. If $p\in (5/3,\infty)$ then for some $\theta\in (0,1)$,
\begin{align*}
    \|\rho(t)\|_{L^p} \les \|\rho(t)\|_{L^{5/3}}^{\theta}\|\rho(t)\|_{L^\infty}^{1-\theta}\les \jbrac{t}^{-3\theta/5}\ln^{6(1-\theta)}\jbrac{t}.
\end{align*}
This converges to zero since $\theta>0$. For $\phi$, if $q<6$ we interpolate between $\|\phi\|_{L^2}$ and $\|\phi(t)\|_{L^6}$, and if $q>6$ we interpolate between $\|\phi(t)\|_{L^6}$ and $\|\phi(t)\|_{L^\infty}$ which both decay as $t\to\infty$.

\printbibliography

\end{document}